\documentclass{amsart}%
\usepackage{graphicx}
\usepackage{amscd}
\usepackage{amsmath}
\usepackage{amsfonts, color}
\usepackage{amssymb}
\usepackage{amsfonts}%
\providecommand{\U}[1]{\protect\rule{.1in}{.1in}}
\providecommand{\U}[1]{\protect\rule{.1in}{.1in}}
\providecommand{\U}[1]{\protect\rule{.1in}{.1in}}
\providecommand{\U}[1]{\protect\rule{.1in}{.1in}}
\providecommand{\U}[1]{\protect\rule{.1in}{.1in}}
\providecommand{\U}[1]{\protect\rule{.1in}{.1in}}
\providecommand{\U}[1]{\protect\rule{.1in}{.1in}}
\newtheorem{theorem}{Theorem}[section]
\theoremstyle{plain}

\newtheorem{corollary}[theorem]{Corollary}

\newtheorem{definition}[theorem]{Definition}

\newtheorem{lemma}[theorem]{Lemma}

\newtheorem{proposition}[theorem]{Proposition}
\newtheorem{remark}[theorem]{Remark}

\numberwithin{equation}{section}

\allowdisplaybreaks[4]
\begin{document}
\title[Duality in Banach and quasi-Banach spaces of homogeneous polynomials]{Duality results in Banach and quasi-Banach spaces of homogeneous polynomials
and applications}
\author{Vin\'{\i}cius V. F\'{a}varo and Daniel Pellegrino }
\thanks{The first named author is supported by FAPESP Grant 2014/50536-7; FAPEMIG
Grant PPM-00086-14; and CNPq Grants 482515/2013-9, 307517/2014-4. }
\thanks{The second named author is supported by CNPq.}

\begin{abstract}
Spaces of homogeneous polynomials on a Banach space are frequently equipped
with quasinorms instead of norms. In this paper we develop a technique to
replace the original quasi-norm by a norm in a dual preserving way, in the
sense that the dual of the space with the new norm coincides with the dual of
the space with the original quasi-norm. Applications to problems on the
existence and approximation of solutions of convolution equations and on
hypercyclic convolution operators on spaces of entire functions are provided.

\end{abstract}
\maketitle

\tableofcontents


%\tableofcontents

\section{Introduction}

%Homogeneous polynomials play a central role in complex analysis.
In the 1960s, many researchers began the study of spaces of holomorphic
functions defined on infinite dimensional complex Banach spaces. In this
context spaces of $n$-homogeneous polynomials play a central role in the
development of the theory. Several tools, such as topological tensor products
and duality theory, are useful and important when we are working with spaces
of homogeneous polynomials. Many duality results on spaces of homogeneous
polynomials and their applications have appeared in the last decades (see for
instance \cite{BBFJ, boyd1, boyd2, boyd3, CDjmaa, CDSjmaa, cp, Dineen-70, DW,
favaro-jatoba1, favaromatospellegrino, FM, G, Gupta, Matos-F, MN} among
others). In this paper we develop a new technique in the duality theory of
spaces of homogeneous polynomials and we give some applications.

Let $E$ be a complex Banach space, $n\in\mathbb{N}$ and $\mathcal{P}(^{n}E)$
be the Banach space of all continuous $n$-homogeneous polynomials from $E$ to
$\mathbb{C}$ with its usual norm. Suppose that $\left(  \mathcal{P}_{\Delta
}(^{n}E),\left\Vert \cdot\right\Vert _{\Delta}\right)  \ $ is a quasi-normed
space of $n$-homogeneous polynomials on $E$ such that the inclusion
$\mathcal{P}_{\Delta}(^{n}E)\hookrightarrow\mathcal{P}(^{n}E)$ is
con\-ti\-nu\-ous and $\mathcal{P}_{f}(^{n}E)\subset\mathcal{P}_{\Delta}%
(^{n}E)$, where $\mathcal{P}_{f}(^{n}E)$ denotes the subspace of
$\mathcal{P}(^{n}E)$ of all polynomials of finite type. Let $C_{\Delta_{n}}>0$
be such that $\left\Vert P\right\Vert \leq C_{\Delta_{n}}\left\Vert
P\right\Vert _{\Delta},$ for all $P\in\mathcal{P}_{\Delta}(^{n}E)$. Suppose
that the normed space $\left(  \mathcal{P}_{\Delta^{\prime}}(^{n}E^{\prime
}),\left\Vert \cdot\right\Vert _{\Delta^{\prime}}\right)  \subset
\mathcal{P}(^{n}E^{\prime})$ is such that the Borel transform%
\[
\mathcal{B}:\left(  \mathcal{P}_{\Delta}(^{n}E)^{\prime},\left\Vert
\cdot\right\Vert \right)  \rightarrow\left(  \mathcal{P}_{\Delta^{\prime}%
}(^{n}E^{\prime}),\left\Vert \cdot\right\Vert _{\Delta^{\prime}}\right)
\]
given by $\mathcal{B}\left(  T\right)  (\varphi)=T(\varphi^{n}),$ for all
$\varphi\in E^{\prime}$ and $T\in\mathcal{P}_{\Delta}(^{n}E)^{\prime},$ is a
topological isomorphism. In this paper we develop a technique to construct a
norm in $\mathcal{P}_{\Delta}(^{n}E)$ in such way that this normed space (or
its completion denoted by $\left(  \mathcal{P}_{\widetilde{\Delta}}\left(
^{n}E\right)  ,\left\Vert \cdot\right\Vert _{\widetilde{\Delta}}\right)  $)
preserves the duality given by the Borel transform, that is the topological
isomorphism given by the Borel transform is still valid when we use
$\mathcal{P}_{\widetilde{\Delta}}\left(  ^{n}E\right) $ instead of
$\mathcal{P}_{\Delta}\left(  ^{n}E\right) $. As applications of this result,
we prove that, under suitable conditions, $\left(  \mathcal{P}_{\widetilde
{\Delta}}\left(  ^{n}E\right)  \right)  _{n=0}^{\infty}$ is a holomorphy type
and we provide new examples of hypercyclic convolution operators and new
existence and approximation results for convolution equations.

The paper is organized as follows:

In Section \ref{sec2} we develop the general theory to obtain a norm in the
quasi-normed space $\mathcal{P}_{\Delta}(^{n}E)$ and to keep the duality via
Borel transform. We also prove some technical results needed to the applications.

In Section \ref{sec3} we present background results that will be needed in the
next section.

In Section \ref{sec4} we obtain the aforementioned applications in a case that
was not possible before. We use as a prototype of model the class of Lorentz
nuclear polynomials.

Throughout the paper $\mathbb{N}$ denotes the set of positive integers and
$\mathbb{N}_{0}$ denotes the set $\mathbb{N}\cup\{0\}$. The letters $E$ and
$F$ will always denote complex Banach spaces and $E^{\prime}$ represents the
topological dual of $E$ and $E^{\prime\prime}$ its bidual. The Banach space of
all continuous $m$-homogeneous polynomials from $E$ into $F$ endowed with its
usual sup norm is denoted by $\mathcal{P}(^{m}E;F)$. The subspace of
$\mathcal{P}(^{m}E;F)$ of all polynomials of finite type is represented by
$\mathcal{P}_{f}(^{m}E;F)$. The linear space of all entire mappings from $E$
into $F$ is denoted by $\mathcal{H}(E;F)$. When $F=\mathbb{C}$ we write
$\mathcal{P}(^{m}E)$, $\mathcal{P}_{f}(^{m}E)$ and $\mathcal{H}(E)$ instead of
$\mathcal{P}(^{m}E;\mathbb{C})$, $\mathcal{P}_{f}(^{m}E;\mathbb{C})$ and
$\mathcal{H}(E;\mathbb{C})$, respectively. For the general theory of
homogeneous polynomials and holomorphic functions we refer to Dineen
\cite{Dineen} and Mujica \cite{Mujica}. If $G$ and $H$ are vector spaces and
$\left\langle \cdot,\cdot\right\rangle $\ is a bilinear form on $G\times H,$
we denote by $(G,H,\left\langle \cdot,\cdot\right\rangle )$ (or $(G,H)$ for
short) the dual system. We denote by $\sigma(G,H)$ the weak topology with
respect to the dual system $(G,H),$ that is, the coarsest topology on $G$ for
which the linear forms $x\rightarrow\left\langle x,y\right\rangle ,$ $y\in H$
are continuous.

\section{Main results\label{sec2}}

Let $n\in\mathbb{N}$ and suppose that $\left(  \mathcal{P}_{\Delta}%
(^{n}E),\left\Vert \cdot\right\Vert _{\Delta}\right)  \ $is a quasi-normed
space of $n$-homogeneous polynomials defined on $E$ such that the inclusion
$\mathcal{P}_{\Delta}(^{n}E)\hookrightarrow\mathcal{P}(^{n}E)$ is continuous
and $\mathcal{P}_{f}(^{n}E)\subset\mathcal{P}_{\Delta}(^{n}E)$. Let
$C_{\Delta_{n}}>0$ be such that $\left\Vert P\right\Vert \leq C_{\Delta_{n}%
}\left\Vert P\right\Vert _{\Delta},$ for all $P\in\mathcal{P}_{\Delta}(^{n}%
E)$. Suppose that the normed space $\left(  \mathcal{P}_{\Delta^{\prime}}%
(^{n}E^{\prime}),\left\Vert \cdot\right\Vert _{\Delta^{\prime}}\right)
\subset\mathcal{P}(^{n}E^{\prime})$ is such that the Borel transform%
\[
\mathcal{B}:\left(  \mathcal{P}_{\Delta}(^{n}E)^{\prime},\left\Vert
\cdot\right\Vert \right)  \rightarrow\left(  \mathcal{P}_{\Delta^{\prime}%
}(^{n}E^{\prime}),\left\Vert \cdot\right\Vert _{\Delta^{\prime}}\right)
\]
given by $\mathcal{B}\left(  T\right)  (\varphi)=T(\varphi^{n}),$ for all
$\varphi\in E^{\prime}$ and $T\in\mathcal{P}_{\Delta}(^{n}E)^{\prime},$ is a
topological isomorphism.

We will show that the pair
\[
\left(  \mathcal{P}_{\Delta}(^{n}E),\mathcal{P}_{\Delta^{\prime}}%
(^{n}E^{\prime})\right)
\]
is a dual system. More precisely, we will prove that there exists a bilinear
form $\left\langle \cdot;\cdot\right\rangle $ on
\[
\mathcal{P}_{\Delta}(^{n}E)\times\mathcal{P}_{\Delta^{\prime}}(^{n}E^{\prime
})
\]
such that the following conditions hold:\newline$\left(  S1\right)  $
$\left\langle P;Q\right\rangle =0$ for all $Q\in\mathcal{P}_{\Delta^{\prime}%
}(^{n}E^{\prime})$ implies $P=0.$\newline$\left(  S2\right)  $ $\left\langle
P;Q\right\rangle =0$ for all $P\in\mathcal{P}_{\Delta}(^{n}E)$ implies $Q=0.$

\bigskip

Let%
\[
\left\langle \cdot,\cdot\right\rangle \colon\mathcal{P}_{\Delta}(^{n}%
E)\times\mathcal{P}_{\Delta^{\prime}}(^{n}E^{\prime})\longrightarrow\mathbb{K}%
\]
be defined by%
\[
\left\langle P;Q\right\rangle =\mathcal{B}^{-1}\left(  Q\right)  \left(
P\right)  .
\]
It is clear that $\left\langle \cdot,\cdot\right\rangle $ is bilinear.

For $0\neq Q\in\mathcal{P}_{\Delta^{\prime}}\left(  ^{n}E^{\prime}\right)  $,
we have $\mathcal{B}^{-1}\left(  Q\right)  \neq0$ (because $\mathcal{B}$ is an
isomorphism). Hence%
\[
\left\langle P;Q\right\rangle =\mathcal{B}^{-1}\left(  Q\right)  \left(
P\right)  \neq0
\]
for some $P\in\mathcal{P}_{\Delta}\left(  ^{n}E\right)  $, and so $\left(
S2\right)  $ holds.

Now, if $0\neq P\in\mathcal{P}_{\Delta}\left(  ^{n}E\right)  ,$ then there is
$x\in E$ such that $P\left(  x\right)  \neq0.$ We consider $A_{x}\in
E^{\prime\prime}$ defined by%
\[
A_{x}\left(  \varphi\right)  =\varphi\left(  x\right)  ,
\]
for all $\varphi\in E^{\prime}$ and define%
\begin{gather*}
T:\mathcal{P}_{\Delta}\left(  ^{n}E\right)  \rightarrow\mathbb{K}\\
T\left(  P\right)  =P\left(  x\right)  .
\end{gather*}
Obviously $T$ is linear. Moreover, $T$ is continuous and $\left\Vert
T\right\Vert \leq C_{\Delta_{n}}\left\Vert x\right\Vert ^{n}$. In fact, for
every $P\in\mathcal{P}_{\Delta}\left(  ^{n}E\right)  $
\[
\left\vert T\left(  P\right)  \right\vert =\left\vert P(x)\right\vert
\leq\left\Vert P\right\Vert \left\Vert x\right\Vert ^{n}\leq C_{\Delta_{n}%
}\left\Vert x\right\Vert ^{n}\left\Vert P\right\Vert _{\Delta}%
\]
and so $\left\Vert T\right\Vert \leq C_{\Delta_{n}}\left\Vert x\right\Vert
^{n}$. Note that%
\[
\mathcal{B}\left(  T\right)  \left(  \varphi\right)  =T(\varphi^{n}%
)=\varphi\left(  x\right)  ^{n}=\left(  A_{x}\left(  \varphi\right)  \right)
^{n}%
\]
for all $\varphi\in E^{\prime}.$ We conclude that the polynomial
\begin{gather*}
\left(  A_{x}\right)  ^{n}:E^{\prime}\rightarrow\mathbb{K}\\
\left(  A_{x}\right)  ^{n}(\varphi)=\varphi(x)^{n}%
\end{gather*}
belongs to $\mathcal{P}_{\Delta^{\prime}}\left(  ^{n}E^{\prime}\right)  $ and%
\[
<P,\left(  A_{x}\right)  ^{n}>=\mathcal{B}^{-1}\left(  \left(  A_{x}\right)
^{n}\right)  \left(  P\right)  =T(P)=P\left(  x\right)  \neq0.
\]
Thus $\left(  S1\right)  $ is proved and hence the pair $\left(
\mathcal{P}_{\Delta}\left(  ^{n}E\right)  ,\mathcal{P}_{\Delta^{\prime}%
}\left(  ^{n}E^{\prime}\right)  \right)  $ is a dual system.

\bigskip

Now, let%
\[
U=\left\{  P\in\mathcal{P}_{\Delta}\left(  ^{n}E\right)  ;\left\Vert
P\right\Vert _{\Delta}\leq1\right\}  .
\]
Since the bipolar of $U$, denoted by $U^{\circ\circ}$, is absorbing we can
consider the corresponding gauge%
\[
p_{U^{\circ\circ}}\left(  P\right)  =\inf\left\{  \delta>0;P\in\delta
U^{\circ\circ}\right\}  ,
\]
defined for all $P$ in $\mathcal{P}_{\Delta}\left(  ^{n}E\right)  .$ Recall
that the polar of $U$ is defined by%
\[
U^{\circ}=\left\{  Q\in\mathcal{P}_{\Delta^{\prime}}\left(  ^{n}E^{\prime
}\right)  ;\left\vert <P,Q>\right\vert \leq1\text{ for all }P\in U\right\}  .
\]
Hence%
\begin{align*}
U^{\circ}  &  =\left\{  Q\in\mathcal{P}_{\Delta^{\prime}}\left(  ^{n}%
E^{\prime}\right)  ;\left\vert \mathcal{B}^{-1}\left(  Q\right)
(P)\right\vert \leq1\text{ for all }P\in\mathcal{P}_{\Delta}\left(
^{n}E\right)  ,\text{ }\left\Vert P\right\Vert _{\Delta}\leq1\right\} \\
&  =\left\{  Q\in\mathcal{P}_{\Delta^{\prime}}\left(  ^{n}E^{\prime}\right)
;\left\Vert \mathcal{B}^{-1}\left(  Q\right)  \right\Vert \leq1\right\}  .
\end{align*}
Moreover,%
\begin{align*}
U^{\circ\circ}  &  =\left\{  P\in\mathcal{P}_{\Delta}\left(  ^{n}E\right)
;\left\vert <P,Q>\right\vert \leq1\text{ for all }Q\in U^{\circ}\right\} \\
&  =\left\{  P\in\mathcal{P}_{\Delta}\left(  ^{n}E\right)  ;\left\vert
\mathcal{B}^{-1}\left(  Q\right)  \left(  P\right)  \right\vert \leq1,\text{
for all }Q\in\mathcal{P}_{\Delta^{\prime}}\left(  ^{n}E^{\prime}\right)
\text{ with }\left\Vert \mathcal{B}^{-1}\left(  Q\right)  \right\Vert
\leq1\right\}
\end{align*}
and hence%
\[
p_{U^{\circ\circ}}\left(  P\right)  =\inf\left\{  \delta>0;\left\vert
\mathcal{B}^{-1}\left(  Q\right)  \left(  P\right)  \right\vert \leq
\delta,\text{ for all }Q\in\mathcal{P}_{\Delta^{\prime}}\left(  ^{n}E^{\prime
}\right)  \text{ with }\left\Vert \mathcal{B}^{-1}\left(  Q\right)
\right\Vert \leq1\right\}  .
\]
Since%
\[
\left\vert \mathcal{B}^{-1}\left(  Q\right)  \left(  P\right)  \right\vert
\leq\left\Vert \mathcal{B}^{-1}\left(  Q\right)  \right\Vert \left\Vert
P\right\Vert _{\Delta}\leq\left\Vert P\right\Vert _{\Delta},
\]
for all $Q\in\mathcal{P}_{\Delta^{\prime}}\left(  ^{n}E^{\prime}\right)  $
with $\left\Vert \mathcal{B}^{-1}\left(  Q\right)  \right\Vert \leq1,$ it
follows that%
\begin{equation}
p_{U^{\circ\circ}}\left(  P\right)  \leq\left\Vert P\right\Vert _{\Delta},
\label{dggg}%
\end{equation}
for all $P\in\mathcal{P}_{\Delta}\left(  ^{n}E\right)  .$

Note that $p_{U^{\circ\circ}}$ is a norm on $\mathcal{P}_{\Delta}\left(
^{n}E\right)  .$ In fact, we only have to prove that $p_{U^{\circ\circ}%
}\left(  P\right)  =0$ implies $P=0.$ If $p_{U^{\circ\circ}}\left(  P\right)
=0,$ then%
\[
\left\vert \mathcal{B}^{-1}\left(  Q\right)  \left(  P\right)  \right\vert =0
\]
for all $Q\in\mathcal{P}_{\Delta^{\prime}}\left(  ^{n}E^{\prime}\right)  $
with $\left\Vert \mathcal{B}^{-1}\left(  Q\right)  \right\Vert \leq1.$ So, we
conclude that
\[
<P,Q>=\left\vert \mathcal{B}^{-1}\left(  Q\right)  \left(  P\right)
\right\vert =0
\]
for all $Q\in\mathcal{P}_{\Delta^{\prime}}\left(  ^{n}E^{\prime}\right)  .$
Hence, from $\left(  S1\right)  $ it follows that $P=0.$

\bigskip

From now on we will often use the Bipolar Theorem, which asserts that the
bipolar of $U$ coincides with the $\sigma\left(  \mathcal{P}_{\Delta}\left(
^{n}E\right)  ,\mathcal{P}_{\Delta^{\prime}}\left(  ^{n}E^{\prime}\right)
\right)  $-closure of the absolutely convex hull $\Gamma\left(  U\right)  $ of
$U$.

\begin{proposition}
\label{desigualdade_gauge}If $P\in\mathcal{P}_{\Delta}\left(  ^{n}E\right)  $
then
\[
\left\Vert P\right\Vert \leq C_{\Delta_{n}}p_{U^{\circ\circ}}\left(  P\right)
.
\]

\end{proposition}

\begin{proof}
We know that%
\[
\left\Vert P\right\Vert \leq C_{\Delta_{n}}\left\Vert P\right\Vert _{\Delta},
\]
for all $P\in\mathcal{P}_{\Delta}\left(  ^{n}E\right)  .$ If $P$ belongs to
the absolutely convex hull $\Gamma\left(  U\right)  $ of $U,$ then%
\[
P=\sum\limits_{j=1}^{m}\lambda_{j}P_{j},
\]
where $P_{j}\in U$, $\lambda_{j}\in\mathbb{K}$, $j=1,\ldots,m,$ for some
$m\in\mathbb{N}$, and%
\[
\sum\limits_{j=1}^{m}\left\vert \lambda_{j}\right\vert \leq1.
\]
Since $P_{j}\in U$, we have%
\[
\left\Vert P_{j}\right\Vert \leq C_{\Delta_{n}}\left\Vert P_{j}\right\Vert
_{\Delta}\leq C_{\Delta_{n}}.
\]
Therefore,%
\begin{equation}
\left\Vert P\right\Vert =\left\Vert \sum\limits_{j=1}^{m}\lambda_{j}%
P_{j}\right\Vert \leq\sum\limits_{j=1}^{m}\left\vert \lambda_{j}\right\vert
\left\Vert P_{j}\right\Vert \leq C_{\Delta_{n}}\sum\limits_{j=1}^{m}\left\vert
\lambda_{j}\right\vert \leq C_{\Delta_{n}}\label{15may}%
\end{equation}
for every $P\in\Gamma\left(  U\right)  $. Now if $P\in U^{\circ\circ},$ which
is the $\sigma\left(  \left(  \mathcal{P}_{\Delta}\left(  ^{n}E\right)
,\mathcal{P}_{\Delta^{\prime}}\left(  ^{n}E^{\prime}\right)  \right)  \right)
$-closure of $\Gamma\left(  U\right)  ,$ let $\left(  P_{i}\right)  _{i\in I}$
be a net in $\Gamma\left(  U\right)  $ such that%
\[
\lim_{i\in I}\left\vert <P_{i},Q>\right\vert =\left\vert <P,Q>\right\vert
\]
for all $Q\in\mathcal{P}_{\Delta^{\prime}}\left(  ^{n}E^{\prime}\right)  .$ So
we have%
\[
\lim_{i\in I}\left\vert \mathcal{B}^{-1}\left(  Q\right)  \left(
P_{i}\right)  \right\vert =\left\vert \mathcal{B}^{-1}\left(  Q\right)
\left(  P\right)  \right\vert
\]
for all $Q\in\mathcal{P}_{\Delta^{\prime}}\left(  ^{n}E^{\prime}\right)  .$ In
particular, for $x\in B_{E},$ we have%
\begin{align*}
\left\vert P\left(  x\right)  \right\vert  &  =\left\vert \mathcal{B}%
^{-1}\left(  \left(  A_{x}\right)  ^{n}\right)  \left(  P\right)  \right\vert
\\
&  =\lim_{i\in I}\left\vert \mathcal{B}^{-1}\left(  \left(  A_{x}\right)
^{n}\right)  \left(  P_{i}\right)  \right\vert =\lim_{i\in I}\left\vert
P_{i}\left(  x\right)  \right\vert \overset{\text{(\ref{15may})}}{\leq
}C_{\Delta_{n}}\left\Vert x\right\Vert .
\end{align*}
Hence%
\begin{equation}
\left\Vert P\right\Vert \leq C_{\Delta_{n}}\label{may15a}%
\end{equation}
for every $P\in U^{\circ\circ}$. Finally, for $0\neq P\in\mathcal{P}_{\Delta
}\left(  ^{n}E\right)  ,$ let%
\[
R=\left(  p_{U^{\circ\circ}}\left(  P\right)  \right)  ^{-1}P.
\]
We thus have%
\[
p_{U^{\circ\circ}}\left(  R\right)  =1,
\]
and this implies that%
\[
\left\vert \mathcal{B}^{-1}\left(  Q\right)  \left(  R\right)  \right\vert
\leq1
\]
for all $Q\in\mathcal{P}_{\Delta^{\prime}}\left(  ^{n}E^{\prime}\right)  $
with $\left\Vert \mathcal{B}^{-1}\left(  Q\right)  \right\Vert \leq1.$ Thus
$R\in U^{\circ\circ}$ and, consequently, from (\ref{may15a}) we conclude that
$\left\Vert R\right\Vert \leq C_{\Delta_{n}},$ i.e.,%
\[
\left\Vert \left(  p_{U^{\circ\circ}}\left(  P\right)  \right)  ^{-1}%
P\right\Vert \leq C_{\Delta_{n}}%
\]
and the result follows.
\end{proof}

\bigskip

We denote the completion of the space $\left(  \mathcal{P}_{\Delta}\left(
^{n}E\right)  ,p_{U^{\circ\circ}}\right)  $ by $\left(  \mathcal{P}%
_{\widetilde{\Delta}}\left(  ^{n}E\right)  ,\left\Vert \cdot\right\Vert
_{\widetilde{\Delta}}\right)  .$ So the restriction of $\left\Vert
\cdot\right\Vert _{\widetilde{\Delta}}$ to $\mathcal{P}_{\Delta}\left(
^{n}E\right)  $ is $p_{U^{\circ\circ}}$ and Proposition
\ref{desigualdade_gauge} implies that
\[
\mathcal{P}_{\widetilde{\Delta}}\left(  ^{n}E\right)  \subset\mathcal{P}%
\left(  ^{n}E\right)
\]
and%
\begin{equation}
\left\Vert P\right\Vert \leq C_{\Delta_{n}}\left\Vert P\right\Vert
_{\widetilde{\Delta}}, \label{desig_norma_usual}%
\end{equation}
for all $P$ in $\mathcal{P}_{\widetilde{\Delta}}\left(  ^{n}E\right)  .$

\begin{definition}
\textrm{\label{def555}The elements of $\mathcal{P}_{\widetilde{\Delta}}\left(
^{n}E\right)  $ are called \emph{quasi-}$\Delta$ $n$-\emph{homogeneous
polynomials}. }
\end{definition}

\begin{remark}
\textrm{When $\mathcal{P}_{\Delta}\left(  ^{n}E\right)  $ is a Banach space
then we have $\left\Vert \cdot\right\Vert _{\widetilde{\Delta}}=\left\Vert
\cdot\right\Vert _{\Delta}$ and $\mathcal{P}_{\widetilde{\Delta}}\left(
^{n}E\right)  =\mathcal{P}_{\Delta}\left(  ^{n}E\right)  .$ }

\textrm{In fact, in this case, $U$ is the closed unit ball in $\mathcal{P}%
_{\Delta}\left(  ^{n}E\right)  $, hence balanced and convex and $\Gamma(U)=U.$
By using the Bipolar Theorem we have%
\[
U^{\circ\circ}=\overline{\Gamma(U)}^{\sigma\left(  \mathcal{P}_{\Delta}\left(
^{n}E\right)  ,\mathcal{P}_{\Delta^{\prime}}\left(  ^{n}E^{\prime}\right)
\right)  }=\overline{U}^{\sigma\left(  \mathcal{P}_{\Delta}\left(
^{n}E\right)  ,\mathcal{P}_{\Delta^{\prime}}\left(  ^{n}E^{\prime}\right)
\right)  }=U,
\]
and the last equality follows from Banach-Mazur Theorem. Hence%
\[
p_{U^{\circ\circ}}\left(  P\right)  =\inf\left\{  \delta>0;P\in\delta
U^{\circ\circ}\right\}  =\inf\left\{  \delta>0;P\in\delta U\right\}
=\inf\left\{  \delta>0;\left\Vert P\right\Vert _{\Delta}\leq\delta\right\}
=\left\Vert P\right\Vert _{\Delta}.
\]
}
\end{remark}

\medskip

The next theorem plays a fundamental role in this work. It assures that the
duals of $\mathcal{P}_{\widetilde{\Delta}}\left(  ^{n}E\right)  $ and
$\mathcal{P}_{\Delta}\left(  ^{n}E\right)  $ are identified by the Borel transform.

\begin{theorem}
The linear mapping%
\begin{align*}
\widetilde{\mathcal{B}}\colon\left(  \mathcal{P}_{\widetilde{\Delta}}\left(
^{n}E\right)  ^{\prime} ,\left\Vert .\right\Vert \right)   &  \longrightarrow
\left(  \mathcal{P}_{\Delta^{\prime}}\left(  ^{n}E^{\prime}\right)
,\left\Vert .\right\Vert _{\Delta^{\prime}}\right) \\
\widetilde{\mathcal{B}}\left(  T\right)  \left(  \varphi\right)   &  =T\left(
\varphi^{n}\right)
\end{align*}
is a topological isomorphism.
\end{theorem}

\begin{proof}
We know that $\left(  \mathcal{P}_{\Delta}\left(  ^{n}E\right)  ,p_{U^{\circ
\circ}}\right)  $ is dense in $\left(  \mathcal{P}_{\widetilde{\Delta}}\left(
^{n}E\right)  ,\left\Vert \cdot\right\Vert _{\widetilde{\Delta}}\right)  .$
Thus the topological duals of both spaces are isometrically isomorphic. So we
only need to prove that $\mathcal{P}_{\Delta}\left(  ^{n}E\right)  $ has the
same topological dual for the norm $p_{U^{\circ\circ}}$ and for the quasi-norm
$\left\Vert \cdot\right\Vert _{\Delta}.$ By (\ref{dggg}), for each
$T\in\left(  \mathcal{P}_{\Delta}\left(  ^{n}E\right)  ,p_{U^{\circ\circ}%
}\right)  ^{\prime}$ we have%
\[
\sup_{P\in U}\left\vert T\left(  P\right)  \right\vert \leq\sup_{p_{U^{\circ
\circ}}\left(  P\right)  \leq1}\left\vert T\left(  P\right)  \right\vert
\]
and this implies that the inclusion
\[
\left(  \mathcal{P}_{\Delta}\left(  ^{n}E\right)  ,p_{U^{\circ\circ}}\right)
^{\prime}\hookrightarrow\left(  \mathcal{P}_{\Delta}\left(  ^{n}E\right)
,\left\Vert \cdot\right\Vert _{\Delta}\right)  ^{\prime}%
\]
is continuous. Now, let $T\in\left(  \mathcal{P}_{\Delta}\left(  ^{n}E\right)
,\left\Vert \cdot\right\Vert _{\Delta}\right)  ^{\prime}.$ If $P\in
\Gamma\left(  U\right)  ,$ then%
\[
P=\sum\limits_{j=1}^{m}\lambda_{j}P_{j},
\]
where $P_{j}\in U$, $\lambda_{j}\in\mathbb{K}$, $j=1,\ldots,m,$ for some
$m\in\mathbb{N}$, and%
\[
\sum\limits_{j=1}^{m}\left\vert \lambda_{j}\right\vert \leq1.
\]
Thus%
\begin{equation}
\left\vert T\left(  P\right)  \right\vert \leq\sum\limits_{j=1}^{m}\left\vert
\lambda_{j}\right\vert \left\Vert T\left(  P_{j}\right)  \right\Vert \leq
\sup_{Q\in U}\left\vert T\left(  Q\right)  \right\vert \sum\limits_{j=1}%
^{m}\left\vert \lambda_{j}\right\vert \leq\sup_{Q\in U}\left\vert T\left(
Q\right)  \right\vert <+\infty.\label{8julhh}%
\end{equation}
If $P\in U^{\circ\circ},$ then there exists a net $\left(  P_{i}\right)
_{i\in I}\subset\Gamma\left(  U\right)  $ such that%
\[
\left\vert T\left(  P\right)  \right\vert =\lim_{i\in I}\left\vert T\left(
P_{i}\right)  \right\vert \overset{\text{(\ref{8julhh})}}{\leq}\sup_{Q\in
U}\left\vert T\left(  Q\right)  \right\vert <+\infty.
\]
Hence $T$ is bounded over $U^{\circ\circ}$ and so continuous for
$p_{U^{\circ\circ}},$ as we wanted to show.
\end{proof}

\bigskip

The next result will be necessary in Section \ref{sec4}. It is clear that
since $\mathcal{P}_{f}\left(  ^{n}E\right)  $ is contained in $\mathcal{P}%
_{\Delta}\left(  ^{n}E\right)  $, then $\mathcal{P}_{f}\left(  ^{n}E\right)  $
is contained in $\mathcal{P}_{\widetilde{\Delta}}\left(  ^{n}E\right)  $.

\begin{proposition}
\label{lema_para_pi1} Let $n\in\mathbb{N}$.\newline(a) If there exists $K>0$
such that $\left\Vert \varphi^{n}\right\Vert _{\Delta}\leq K\left\Vert
\varphi\right\Vert ^{n},$ for all $\varphi\in E^{\prime}$, then%
\[
\left\Vert \varphi^{n}\right\Vert _{\widetilde{\Delta}}\leq K\left\Vert
\varphi\right\Vert ^{n}\leq KC_{\Delta_{1}}^{n}\left\Vert \varphi\right\Vert
_{\widetilde{\Delta}}^{n}%
\]
for all $\varphi\in E^{\prime}.$\newline(b) If $\mathcal{P}_{f}\left(
^{n}E\right)  $ is dense in $\left(  \mathcal{P}_{\Delta}\left(  ^{n}E\right)
,\left\Vert \cdot\right\Vert _{\Delta}\right)  ,$ then $\mathcal{P}_{f}\left(
^{n}E\right)  $ is dense in $\left(  \mathcal{P}_{\widetilde{\Delta}}\left(
^{n}E\right)  ,\left\Vert \cdot\right\Vert _{\widetilde{\Delta}}\right)  .$
\end{proposition}

\begin{proof}
(a) By inequality (\ref{dggg}) we have $\left\Vert P\right\Vert _{\widetilde
{\Delta}}\leq\left\Vert P\right\Vert _{\Delta}$, for every $P\in
\mathcal{P}_{\Delta}\left(  ^{n}E\right)  $. In particular, for every
$\varphi\in E^{\prime}$,
\[
\left\Vert \varphi^{n}\right\Vert _{\widetilde{\Delta}}\leq\left\Vert
\varphi^{n}\right\Vert _{\Delta}\leq K\left\Vert \varphi\right\Vert ^{n}.
\]
Besides, (\ref{desig_norma_usual}) assures that
\[
\left\Vert \varphi\right\Vert \leq C_{\Delta_{1}}\left\Vert \varphi\right\Vert
_{\widetilde{\Delta}}
\]
for every $\varphi\in E^{\prime}$. Now the result follows from the last two inequalities.

(b) We know that $p_{U^{\circ\circ}}\left(  \cdot\right)  \leq\left\Vert
\cdot\right\Vert _{\Delta}$(see (\ref{dggg})) and $p_{U^{\circ\circ}}\left(
P\right)  =\left\Vert P\right\Vert _{\widetilde{\Delta}}$, for all
$P\in\mathcal{P}_{\Delta}\left(  ^{n}E\right)  $. Using this fact and the
density of $\mathcal{P}_{f}\left(  ^{n}E\right)  $ in $\left(  \mathcal{P}%
_{\Delta}\left(  ^{n}E\right)  ,\left\Vert \cdot\right\Vert _{\Delta}\right)
,$ the result follows.
\end{proof}

\medskip

Now we are interested in connecting the previous construction with the concept
of holomorphy type that we recall below. The notation for the derivatives of
polynomials that we use are the same introduced by L. Nachbin in
\cite{nachbin}.

\begin{definition}
\textrm{\textrm{\textrm{\label{Definition tipo holomorfia} A \emph{holomorphy
type} $\Theta$ from $E$ to $F$ is a sequence of Banach spaces $(\mathcal{P}%
_{\Theta}(^{n}E;F))_{n=0}^{\infty}$, the norm on each of them being denoted by
$\|\cdot\|_{\Theta}$, such that the following conditions hold true: } } }

\item[$(1)$] \textrm{\textrm{\textrm{Each $\mathcal{P}_{\Theta}(^{n}E;F)$ is a
linear subspace of $\mathcal{P}(^{n}E;F)$. } } }

\item[$(2)$] \textrm{\textrm{\textrm{$\mathcal{P}_{\Theta}(^{0}E;F)$ coincides
with $\mathcal{P}(^{0}E;F)=F$ as a normed vector space.} } }

\item[$(3)$] \textrm{\textrm{\textrm{There is a real number $\sigma\geq1$ for
which the following is true: given any $k\in\mathbb{N}_{0}$, $n\in
\mathbb{N}_{0}$, $k\leq n$, $a\in E$ and $P\in\mathcal{P}_{\Theta}(^{n}E;F)$,
we have
\[
\hat{d}^{k}P(a)\in\mathcal{P}_{\Theta}(^{k}E;F)~~and
\]%
\[
\left\Vert \frac{1}{k!}\hat{d}^{k}P(a)\right\Vert _{\Theta}\leq\sigma^{n}\Vert
P\Vert_{\Theta}\Vert a\Vert^{n-k}.
\]
} } }
\end{definition}

\noindent It is plain that each inclusion $\mathcal{P}_{\Theta}(^{n}%
E;F)\subseteq\mathcal{P}(^{n}E;F)$ is continuous and that $\Vert P\Vert
\leq\sigma^{n}\Vert P\Vert_{\Theta}$ for every $P\in\mathcal{P}_{\Theta}%
(^{n}E;F)$.

\bigskip

The definition of holomorphy type motivates the next definition for
quasi-normed spaces of homogeneous polynomials.

\begin{definition}
\textrm{\textrm{\label{Def_stability_for_derivatives} For each $n\in
\mathbb{N}_{0}$, let $\left(  \mathcal{P}_{\Delta}\left(  ^{n}E\right)
,\left\Vert \cdot\right\Vert _{\Delta}\right)  $ be a quasi-normed space,
where $\mathcal{P}_{\Delta}\left(  ^{0}E\right)  =\mathbb{C}$. The sequence
$\left(  \mathcal{P}_{\Delta}\left(  ^{n}E\right)  \right)  _{n=0}^{\infty}$
is \emph{stable for derivatives }if } }

\textrm{\textrm{(1) $\hat{d}^{k}P\left(  x\right)  \in\mathcal{P}_{\Delta
}\left(  ^{k}E\right)  $ for each $n\in\mathbb{N}_{0}$, $P\in\mathcal{P}%
_{\Delta}\left(  ^{n}E\right)  ,$ $k=0,1,\ldots,n$ and $x\in E.$ } }

\textrm{\textrm{(2) For each $n\in\mathbb{N}_{0}$, $k=0,1,\ldots,n,$ there is
a constant $C_{n,k}\geq0$ such that%
\[
\left\Vert \hat{d}^{k}P\left(  x\right)  \right\Vert _{\Delta}\leq
C_{n,k}\left\Vert P\right\Vert _{\Delta}\left\Vert x\right\Vert ^{n-k},
\]
for all $x\in E.$ } }
\end{definition}

\begin{theorem}
\label{Theo_stability_for_derivatives}Let $\left(  \mathcal{P}_{\Delta}\left(
^{n}E\right)  \right)  _{n=0}^{\infty}$ be a sequence stable for
derivatives.\emph{ }If $P\in\mathcal{P}_{\widetilde{\Delta}}\left(
^{n}E\right)  $, then
\[
\hat{d}^{k}P(x)\in\mathcal{P}_{\widetilde{\Delta}}\left(  ^{k}E\right)
\]
and%
\[
\left\Vert \hat{d}^{k}P(x)\right\Vert _{\widetilde{\Delta}}\leq C_{n,k}%
\left\Vert P\right\Vert _{\widetilde{\Delta}}\left\Vert x\right\Vert ^{n-k},
\]
for every $k=0,1,\ldots,n$ and $x\in E$, where $C_{n,k}$ is the constant of
Definition \ref{Def_stability_for_derivatives}.
\end{theorem}

\begin{proof}
By hypothesis we have%
\begin{equation}
\left\Vert \hat{d}^{k}P\left(  x\right)  \right\Vert _{\Delta}\leq
C_{n,k}\left\Vert P\right\Vert _{\Delta}\left\Vert x\right\Vert ^{n-k},
\label{9jjj}%
\end{equation}
for all $P\in\mathcal{P}_{\Delta}\left(  ^{n}E\right)  $, $k=0,1,\ldots,n$ and
$x\in E.$ Let $U_{k}=\left\{  Q\in\mathcal{P}_{\Delta}\left(  ^{k}E\right)
;\left\Vert Q\right\Vert _{\Delta}\leq1\right\}  $ and let $V_{k}$ be the
absolutely convex hull of $U_{k}.$ Let $p_{V_{k}}$ be the gauge of $V_{k}$
(note that $p_{V_{k}}$ is a norm, since $V_{k}$ is a bounded, balanced and
convex neighborhood of zero). Consider%
\begin{gather*}
\psi:\mathcal{P}_{\Delta}\left(  ^{n}E\right)  \rightarrow\mathcal{P}_{\Delta
}\left(  ^{k}E\right) \\
\psi(P)=\hat{d}^{k}P(x).
\end{gather*}
We know that%
\begin{equation}
p_{V_{k}}(Q)\leq\left\Vert Q\right\Vert _{\Delta} \label{nnbv}%
\end{equation}
for every $Q\in\mathcal{P}_{\Delta}\left(  ^{k}E\right)  $. In fact, since
$V_{k}$ is convex, balanced and absorbing, we have%
\begin{equation}
V_{k}\subset\left\{  Q\in\mathcal{P}_{\Delta}\left(  ^{k}E\right)  ;p_{V_{k}%
}(Q)\leq1\right\}  \label{duas-es}%
\end{equation}
and, for $Q\in\mathcal{P}_{\Delta}\left(  ^{k}E\right)  $, $Q\neq0,$ we have%
\[
\left\Vert \frac{Q}{\left\Vert Q\right\Vert _{\Delta}}\right\Vert _{\Delta
}=1.
\]
Hence,%
\[
\frac{Q}{\left\Vert Q\right\Vert _{\Delta}}\in U_{k}\subset V_{k}%
\]
and
\[
p_{V_{k}}\left(  \frac{Q}{\left\Vert Q\right\Vert _{\Delta}}\right)  \leq1,
\]
which shows (\ref{nnbv}). From (\ref{nnbv}) we get%
\[
p_{V_{k}}(\hat{d}^{k}P(x))\leq\left\Vert \hat{d}^{k}P(x)\right\Vert _{\Delta
}\overset{\text{(\ref{9jjj})}}{\leq}C_{n,k}\left\Vert P\right\Vert _{\Delta
}\left\Vert x\right\Vert ^{n-k}%
\]
for every $P\in\mathcal{P}_{\Delta}\left(  ^{n}E\right)  .$ Now let $Q\in
V_{n}$. Then
\[
Q=%
%TCIMACRO{\dsum \limits_{j=1}^{m}}%
%BeginExpansion
{\displaystyle\sum\limits_{j=1}^{m}}
%EndExpansion
\lambda_{j}P_{j}%
\]
with $P_{j}\in U_{n},$ $j=1,\ldots,m$ and $\left\vert \lambda_{1}\right\vert
+\cdots+\left\vert \lambda_{m}\right\vert =1.$ Hence%
\begin{align}
p_{V_{k}}(\hat{d}^{k}Q(x))  &  \leq%
%TCIMACRO{\dsum \limits_{j=1}^{m}}%
%BeginExpansion
{\displaystyle\sum\limits_{j=1}^{m}}
%EndExpansion
\left\vert \lambda_{j}\right\vert p_{V_{k}}\left(  \hat{d}^{k}P_{j}(x)\right)
\label{lkjh}\\
&  \leq C_{n,k}\left\Vert P_{j}\right\Vert _{\Delta}\left\Vert x\right\Vert
^{n-k}\nonumber\\
&  \leq C_{n,k}\left\Vert x\right\Vert ^{n-k}.\nonumber
\end{align}
for every $Q\in V_{n}$.

If $P\in\mathcal{P}_{\Delta}\left(  ^{n}E\right)  $, $P\neq0,$ then for every
$\varepsilon>0$ we have%
\[
p_{V_{n}}\left(  \frac{P}{p_{V_{n}}(P)+\varepsilon}\right)  <1
\]
and it follows from the definition of $p_{V_{n}}$ that%
\[
\frac{P}{p_{V_{n}}(P)+\varepsilon}\in1V_{n}=V_{n}.
\]
Hence, from (\ref{lkjh}) we have%
\[
p_{V_{k}}\left(  \hat{d}^{k}\left(  \frac{P}{p_{V_{n}}(P)+\varepsilon}\right)
(x)\right)  \leq C_{n,k}\left\Vert x\right\Vert ^{n-k}%
\]
for every $\varepsilon>0$. Since $\varepsilon>0$ is arbitrary, we obtain%
\begin{equation}
p_{V_{k}}(\hat{d}^{k}P(x))\leq C_{n,k}\left\Vert x\right\Vert ^{n-k}p_{V_{n}%
}(P).\label{quatro-es}%
\end{equation}
From the Bipolar Theorem we know that $U_{n}^{\circ\circ}$ is the weak closure
of $V_{n}$. It is clear that (\ref{duas-es}) holds for $n$ in the place of
$k,$ so we also have%
\begin{equation}
V_{n}\subset\left\{  Q\in\mathcal{P}_{\Delta}\left(  ^{n}E\right)  ;p_{V_{n}%
}(Q)\leq1\right\}  .\nonumber
\end{equation}

Note that $\left(  \mathcal{P}_{\Delta}\left(  ^{n}E\right)  ,p_{V_{n}%
}\right)  $ is consistent with the dual system $\left(  \mathcal{P}_{\Delta
}\left(  ^{n}E\right)  ,\mathcal{P}_{\Delta^{\prime}}\left(  ^{n}E^{\prime
}\right)  \right)  , $ and this means that the dual of $\mathcal{P}_{\Delta
}\left(  ^{n}E\right)  $ endowed with the topologies $p_{V_{n}}$ and $\left\Vert
\cdot\right\Vert _{\Delta}$ is the same. In fact,
\[
p_{U_{n}^{\circ\circ}}\leq p_{V_{n}}\leq p_{U_{n}}\leq\left\Vert
\cdot\right\Vert _{\Delta}%
\]
and $\left(  \mathcal{P}_{\Delta}\left(  ^{n}E\right)  ,p_{U_{n}^{\circ\circ}%
}\right)  $ is dense in $\left(  \mathcal{P}_{\widetilde{\Delta}}\left(
^{n}E\right)  ,\left\Vert \cdot\right\Vert _{\tilde{\Delta}}\right)  . $ Hence%
\begin{align*}
\left(  \mathcal{P}_{\Delta}\left(  ^{n}E\right)  ,p_{U_{n}^{\circ\circ}%
}\right)  ^{\prime}  &  =\left(  \mathcal{P}_{\widetilde{\Delta}}\left(
^{n}E\right)  ,\left\Vert \cdot\right\Vert _{\widetilde{\Delta}}\right)
^{\prime}\\
&  =\left(  \mathcal{P}_{\Delta^{\prime}}(^{n}E^{\prime}),\left\Vert
\cdot\right\Vert _{\Delta^{\prime}}\right) \\
&  =\left(  \mathcal{P}_{\Delta}\left(  ^{n}E\right)  ,\left\Vert
\cdot\right\Vert _{\Delta}\right)  ^{\prime}.
\end{align*}

From \cite[3.1 p. 130]{SHA} we know that the closure of a convex set is the
same no matter how we choose the topology (consistent with the dual system).
Since $V_{n}$ is convex, we have
\begin{align*}
\left\{  Q\in\mathcal{P}_{\Delta}\left(  ^{n}E\right)  ;p_{V_{n}}%
(Q)\leq1\right\}   &  =\overline{\left\{  Q\in\mathcal{P}_{\Delta}\left(
^{n}E\right)  ;p_{V_{n}}(Q)\leq1\right\}  }^{p_{V_{n}}}\\
&  =\overline{\left\{  Q\in\mathcal{P}_{\Delta}\left(  ^{n}E\right)
;p_{V_{n}}(Q)\leq1\right\}  }^{\sigma\left(  \mathcal{P}_{\Delta}\left(
^{n}E\right)  ,\mathcal{P}_{\Delta^{\prime}}\left(  ^{n}E^{\prime}\right)
\right)  }.
\end{align*}
Hence%
\begin{equation}
U_{n}^{\circ\circ}=\overline{V_{n}}^{\sigma\left(  \mathcal{P}_{\Delta}\left(
^{n}E\right)  ,\mathcal{P}_{\Delta^{\prime}}\left(  ^{n}E^{\prime}\right)
\right)  }=\left\{  Q\in\mathcal{P}_{\Delta}\left(  ^{n}E\right)  ;p_{V_{n}%
}(Q)\leq1\right\}  . \label{cinco-es}%
\end{equation}
Thus we have%
\begin{align}
p_{U_{k}^{\circ\circ}}(\psi(P))  &  \leq p_{V_{k}}(\psi(P))\label{seis-es}\\
&  \overset{\text{(\ref{quatro-es})}}{\leq}C_{n,k}\left\Vert x\right\Vert
^{n-k}p_{V_{n}}(P)\nonumber\\
&  \overset{\text{(\ref{cinco-es})}}{\leq}C_{n,k}\left\Vert x\right\Vert
^{n-k}\nonumber
\end{align}
for every $P\in U_{n}^{\circ\circ}.$

Now, let $P\in\mathcal{P}_{\Delta}\left(  ^{n}E\right)  $, $P\neq0$. From the
argument used just after (\ref{may15a}) we have%
\[
\frac{P}{p_{U_{n}^{\circ\circ}}(P)}\in U_{n}^{\circ\circ}%
\]
and hence%
\[
p_{U_{k}^{\circ\circ}}\left(  \hat{d}^{k}\left(  \frac{P}{p_{U_{n}^{\circ
\circ}}(P)}\right)  (x)\right)  \overset{\text{(\ref{seis-es})}}{\leq}%
C_{n,k}\left\Vert x\right\Vert ^{n-k}%
\]
and we finally conclude that
\[
p_{U_{k}^{\circ\circ}}\left(  \psi(P)\right)  \leq C_{n,k}\left\Vert
x\right\Vert ^{n-k}p_{U_{n}^{\circ\circ}}(P).
\]
This implies that $\psi$ is a continuous linear mapping from $\left(
\mathcal{P}_{\Delta}\left(  ^{n}E\right)  ,p_{U_{n}^{\circ\circ}}\right)  $
into $\left(  \mathcal{P}_{\Delta}\left(  ^{k}E\right)  ,p_{U_{k}^{\circ\circ
}}\right)  .$ We can now extend $\psi$ to the completions $\left(
\mathcal{P}_{\Delta}\left(  ^{n}E\right)  ,\left\Vert \cdot\right\Vert
_{\widetilde{\Delta}}\right)  $ and $\left(  \mathcal{P}_{\Delta}\left(
^{k}E\right)  ,\left\Vert \cdot\right\Vert _{\widetilde{\Delta}}\right)  $ and
the proof is done.
\end{proof}

\begin{corollary}
\label{corollary_holom_type}If $\left(  \mathcal{P}_{\Delta}\left(
^{n}E\right)  \right)  _{n=0}^{\infty}$ is stable for derivatives with
$C_{n,k}\leq\frac{n!}{\left(  n-k\right)  !},$ then $\left(  \mathcal{P}%
_{\widetilde{\Delta}}\left(  ^{n}E\right)  \right)  _{n=0}^{\infty}$ is a
holomorphy type.
\end{corollary}

\begin{proof}
Since conditions $(1)$ and $(2)$ of Definition
\ref{Definition tipo holomorfia} are clear, we only have to prove $(3).$ We
will show that for $\sigma=2,$ we obtain $(3)$. Let $P\in\mathcal{P}%
_{\widetilde{\Delta}}\left(  ^{n}E\right)  ,$ $k\in\mathbb{N}_{0},$\textrm{
}$k\leq n$ and $x\in E.$ By Theorem \ref{Theo_stability_for_derivatives} we
have $\hat{d}^{k}P\left(  x\right)  \in\mathcal{P}_{\widetilde{\Delta}}\left(
^{k}E\right)  $ and%
\[
\left\Vert \hat{d}^{k}P\left(  x\right)  \right\Vert _{\widetilde{\Delta}}%
\leq\frac{n!}{\left(  n-k\right)  !}\left\Vert P\right\Vert _{\widetilde
{\Delta}}\left\Vert x\right\Vert ^{n-k}.
\]
Hence%
\[
\left\Vert \frac{1}{k!}\hat{d}^{k}P\left(  x\right)  \right\Vert
_{\widetilde{\Delta}}\leq\frac{n!}{k!\left(  n-k\right)  !}\left\Vert
P\right\Vert _{\widetilde{\Delta}}\left\Vert x\right\Vert ^{n-k}\leq
2^{n}\left\Vert P\right\Vert _{\widetilde{\Delta}}\left\Vert x\right\Vert
^{n-k},
\]
as we wanted to show.
\end{proof}

Corollary \ref{corollary_holom_type} tells us how to use the results of this
section to obtain a holomorphy type $\left(  \mathcal{P}_{\widetilde{\Delta}%
}\left(  ^{n}E\right)  \right)  _{n=0}^{\infty}$ from $\left(  \mathcal{P}%
_{\Delta}\left(  ^{n}E\right)  \right)  _{n=0}^{\infty}.$ This result will be
useful in Section \ref{sec4}.

\section{\label{sec3} Prerequisites for the applications}

We are interested in applying the results of the previous section to obtain
new examples of hypercyclic convolution operators and new existence and
approximation results for convolution equations. We start presenting a little
of the state of the art of both topics.

If $X$ is a topological space, a map $f:X\rightarrow X$ is \emph{hypercyclic}
if the set $\{x,f(x),f^{2}(x),\ldots\}$ is dense in $X$ for some $x\in X$. In
this case, $x$ is said to be a \emph{hypercyclic vector for $f$}. Hypercyclic
translation and differentiation operators on spaces of entire functions of one
complex variable were first investigated by Birkhoff \cite{birkhoff} and
MacLane \cite{maclane}. Godefroy and Shapiro \cite{godefroy} pushed these
results quite further by proving that every convolution operator on spaces of
entire functions of several complex variables which is not a scalar multiple
of the identity is hypercyclic. For the theory of hypercyclic operators and
its ramifications we refer to \cite{bay, BPS1, goswinBAMS} and references
therein. We remark that several results on the hypercyclicity of operators on
spaces of entire functions on infinitely many complex variables appeared later
(see, e.g., \cite{bay2, BBFJ, bes2012, CDSjmaa, chan, gethner, goswinBAMS,
MPS, peterssonjmaa}). In 2007, Carando, Dimant and Muro \cite{CDSjmaa} proved
some general results, including a solution to a problem posed in
\cite{aronmarkose}, that encompass as particular cases several of the above
mentioned results. In \cite{BBFJ}, using the theory of holomorphy types,
Bertoloto, Botelho, F\'{a}varo and Jatob\'{a} generalized the results of
\cite{CDSjmaa} to a more general setting. For instance, the following theorem
from \cite{BBFJ}, when restricted to $E=\mathbb{C}^{n}$ and $\mathcal{P}%
_{\Theta}(^{m}\mathbb{C}^{n})=\mathcal{P}(^{m}\mathbb{C}^{n})$ recovers the
famous result of Godefroy and Shapiro \cite{godefroy} on the hypercyclicity of
convolution operators on $\mathcal{H}(\mathbb{C}^{n})$:

\medskip

\textbf{Theorem } \cite[Theorem 2.7]{BBFJ} Let $E^{\prime}$ be separable and
$(\mathcal{P}_{\Theta}(^{m}E))_{m=0}^{\infty}$ be a $\pi_{1}$-holomorphy type
from $E$ to $\mathbb{C}$. Then every convolution operator on $\mathcal{H}%
_{\Theta b}(E)$ which is not a scalar multiple of the identity is hypercyclic.

\medskip

However, the spaces $\mathcal{P}_{\Theta}(^{m}E)$ need to be Banach spaces and
thus $\mathcal{H}_{\Theta b}(E)$ becomes a Fr\'{e}chet space. When the spaces
$\mathcal{P}_{\Theta}(^{m}E)$ are quasi-Banach, the respective space
$\mathcal{H}_{\Theta b}(E)$ is not Fr\'{e}chet and then some arguments used to
prove the result above do not work.

Also, several spaces of holomorphic mappings, which have arisen with the
development of the theory of polynomial and operator ideals (see, for instance
\cite{defant, Pietsch, pietsch}), are not Fr\'{e}chet spaces and thus the
investigation of this more general setting seems to be relevant.

The same problem happens in the investigation of existence and approximation
results for convolution equations. This line of investigation was initiated by
Malgrange \cite{Malgrange} and developed by several authors (see, for instance
\cite{Col-Mat, cp, cgp, DW, DW2, Fa, FaBelg, favaro-jatoba1, favaro-jatoba2,
G, Gupta, Martineau, Matos-F, Matos-Z, Matos-Z2, Matos-livro, MN, Nach-B}). In
this context, a result of \cite{favaro-jatoba1} (refined in \cite{BBFJ}) gives
a general method to prove existence and approximation results for convolution
equations defined on certain spaces of entire functions of bounded type. The
general process to prove existence and approximation results for convolution
equations on $\mathcal{H}_{\Theta b}(E)$ involves three main steps:

(i) To establish an isomorphism between the topological dual of $\mathcal{H}%
_{\Theta b}(E)$ and a certain space $\mathcal{E}$ of exponential-type
holomorphic functions via Borel transform.

(ii) To prove a division theorem for holomorphic functions on $\mathcal{E}$,
that is, if $fg=h$, $g\ne0$, $g,h\in\mathcal{E}$, $f\in\mathcal{H}(E^{\prime
})$, then it is possible to show that $f\in\mathcal{E}$.

(iii) To handle the results of (i) and (ii) and use Hahn--Banach type theorems
and the Dieudonn\'e--Schwartz theorem that appears in \cite{DS}.

The absence of Hahn--Banach and Dieudonn\'{e}--Schwartz theorems for more
general settings is a crucial obstacle for the development of a general
theory. For some classes of polynomials $\mathcal{P}_{\Theta}(^{m}E) $ there
are duality results via Borel transform but the step (iii) can not be accomplished.

The results of Section 1, together with the results of \cite{BBFJ}, allow us
to deal with these problematic cases of hypercyclicity and existence and
approximation results for convolution equations. We shall use the results of
Section 1 in such way that steps (i)-(iii) are applicable. It is worth
mentioning that the proof of division theorems (step (ii)) for holomorphic
functions is always a hardwork (see, e.g., \cite{cp, cgp,FaBelg, hhh, loja,
Martineau, Matos-C2}).

Now we present some preliminary results for the applications.

\begin{definition}
\textrm{\textrm{\cite[Definition 2.2]{favaro-jatoba1}%
\label{Definition f holomorfia} \textrm{Let $(\mathcal{P}_{\Theta}%
(^{m}E;F))_{m=0}^{\infty}$ be a holomorphy type from $E$ to $F$. A given
$f\in\mathcal{H}(E;F)$ is said to be of \emph{$\Theta$-holomorphy type of
bounded type} if } } }

\item[$(i)$] \textrm{\textrm{\textrm{$\hat{d}^{m}f(0)\in\mathcal{P}_{\Theta
}(^{m}E;F)$, for all $m\in\mathbb{N}_{0}$, } } }

\item[$(ii)$] \textrm{\textrm{\textrm{$\lim_{m\rightarrow\infty}\left(
\frac{1}{m!}\Vert\hat{d}^{m}f(0)\Vert_{\Theta}\right)  ^{\frac{1}{m}}=0.$ } }
}

\textrm{\textrm{\textrm{The vector subspace of $\mathcal{H}(E;F)$ of all such
$f$ is denoted by $\mathcal{H}_{\Theta b}(E;F)$ and becomes a Fr\'{e}chet
space with the topology $\tau_{\Theta}$ generated by the family of seminorms
\[
f\in\mathcal{H}_{\Theta b}(E;F)\mapsto\Vert f\Vert_{\Theta,\rho}=\sum
_{m=0}^{\infty}\frac{\rho^{m}}{m!}\Vert\hat{d}^{m}f(0)\Vert_{\Theta},
\]
for all $\rho>0$ (see \cite[Proposition 2.3]{favaro-jatoba1}). } } }

\textrm{\textrm{\textrm{When $F=\mathbb{C}$ we represent $\mathcal{H}_{\Theta
b}(E;\mathbb{C}):=\mathcal{H}_{\Theta b}(E).$ } } }
\end{definition}

The next two definitions are slight variations of the concepts of $\pi_{1}$
and $\pi_{2}$ holomorphy types (originally introduced in \cite{favaro-jatoba1}%
) and they can be found in \cite{BBFJ}.

\begin{definition}
\textrm{\textrm{\textrm{\label{pi-tipo de holomorfia} A holomorphy type
$(\mathcal{P}_{\Theta}(^{m}E;F))_{m=0}^{\infty}$ from $E$ to $F$ is said to be
a \emph{$\pi_{1}$-holomorphy type} if the following conditions hold: } } }

\textrm{\textrm{(i)\textrm{ Polynomials of finite type belong to
$(\mathcal{P}_{\Theta}(^{m}E;F))_{m=0}^{\infty}$ and there exists $K>0$ such
that
\[
\Vert\phi^{m}\cdot b\Vert_{\Theta}\leq K^{m}\Vert\phi\Vert^{m}\cdot\Vert
b\Vert
\]
for all $\phi\in E^{\prime}$, $b\in F$ and $m\in\mathbb{N}$; } } }

\textrm{\textrm{(ii)\textrm{ For each $m\in\mathbb{N}_{0}$, $\mathcal{P}%
_{f}(^{m}E;F)$ is dense in $(\mathcal{P}_{\Theta}(^{m}E;F),\Vert\cdot
\Vert_{\Theta})$. } } }
\end{definition}

\begin{definition}
\textrm{\textrm{\textrm{\noindent\ A holomorphy type $(\mathcal{P}_{\Theta
}(^{m}E))_{m=0}^{\infty}$ from $E$ to $\mathbb{C}$ is said to be a
\emph{$\pi_{2}$-holomorphy type} if for each $T\in\left[  \mathcal{H}_{\Theta
b}(E)\right]  ^{\prime}$, $m\in\mathbb{N}_{0}$ and $k\in\mathbb{N}_{0},$
$k\leq m$, the following conditions hold: } } }

\textrm{\textrm{(i)\textrm{ If $P\in\mathcal{P}_{\Theta}(^{m}E)$ and $A\colon
E^{m}\longrightarrow\mathbb{C}$ is the unique continuous symmetric $m$-linear
mapping such that $P=\hat{A},$ then the $\left(  m-k\right)  $-homogeneous
polynomial%
\begin{align*}
T\left(  \widehat{A(\cdot)^{k}}\right)  \colon E  &  \longrightarrow
\mathbb{C}\\
y  &  \mapsto T\left(  A(\cdot)^{k}y^{m-k}\right)
\end{align*}
belongs to $\mathcal{P}_{\Theta}(^{m-k}E);$ } } }

\textrm{\textrm{(ii)\textrm{ For constants $C,\rho>0$ such that
\[
\left\vert T\left(  f\right)  \right\vert \leq C\left\Vert f\right\Vert
_{\Theta,\rho}~for~every~f\in\mathcal{H}_{\Theta b}(E)
\]
(which exist since $T\in\left[  \mathcal{H}_{\Theta b}(E)\right]  ^{\prime}),$
there is a constant }$K>0$\textrm{ such that
\[
\Vert T(\widehat{A(\cdot)^{k}})\Vert_{\Theta}\leq C\cdot K^{m}\rho^{k}\Vert
P\Vert_{\Theta}~for~every~P\in\mathcal{P}_{\Theta}(^{m}E).
\]
} } }
\end{definition}

When we write ``$\Theta$ is a $\pi_{1}$-$\pi_{2}$-holomorphy type'', it means
that $\Theta$ is a $\pi_{1}$ and a $\pi_{2}$-holomorphy type.

\medskip

Let $\Theta$ be a $\pi_{1}$-holomorphy type from $E$ to $F.$ It is clear that
the Borel transform
\[
\mathcal{B}_{\Theta}\colon\left[  \mathcal{P}_{\Theta}(^{m}E;F)\right]
^{\prime}\longrightarrow\mathcal{P}(^{m}E^{\prime};F^{\prime})~,~\mathcal{B}%
_{\Theta}T(\phi)(y)=T(\phi^{m}y),
\]
for $T\in\left[  \mathcal{P}_{\Theta}(^{m}E;F)\right]  ^{\prime}$, $\phi\in
E^{\prime}$ and $y\in F$, is well-defined and linear. Moreover, by (i) and
(ii) of Definition \ref{pi-tipo de holomorfia}, $\mathcal{B}_{\Theta}$ is
continuous and injective. So, denoting the range of $\mathcal{B}_{\Theta}$ in
$\mathcal{P}(^{m}E^{\prime};F^{\prime})$ by $\mathcal{P}_{\Theta^{\prime}%
}(^{m}E^{\prime};F^{\prime})$, the correspondence
\[
\mathcal{B}_{\Theta}T\in\mathcal{P}_{\Theta^{\prime}}(^{m}E^{\prime}%
;F^{\prime})\mapsto\Vert\mathcal{B}_{\Theta}T\Vert_{\Theta^{\prime}}:=\Vert
T\Vert
\]
defines a norm on $\mathcal{P}_{\Theta^{\prime}}(^{m}E^{\prime};F^{\prime})$.

In this fashion the spaces $\left(  \left[  \mathcal{P}_{\Theta}%
(^{m}E;F)\right]  ^{\prime}\;,\Vert\cdot\Vert\right)  $ and $\left(
\mathcal{P}_{\Theta^{\prime}}(^{m}E^{\prime};F^{\prime}),\;\Vert\cdot
\Vert_{\Theta^{\prime}}\right)  $ are isometrically isomorphic. For more
details on this isomorphism we refer \cite{BBFJ} or \cite{favaro-jatoba1}.

\begin{definition}
\textrm{\textrm{\cite[Definition 2.6]{BBFJ}\textrm{ Let $\Theta$ be a
holomorphy type from $E$ to $\mathbb{C}$. \newline(a) For $a\in E$ and
$f\in\mathcal{H}_{\Theta b}(E)$, the \textit{translation of $f$ by $a$} is
defined by
\[
\tau_{a}f\colon E\longrightarrow\mathbb{C}~,~\left(  \tau_{a}f\right)  \left(
x\right)  =f\left(  x-a\right)  .
\]
By \cite[Proposition 2.2]{favaro-jatoba1} we have $\tau_{a}f\in\mathcal{H}%
_{\Theta b}(E).$\newline(b) A continuous linear operator $L\colon
\mathcal{H}_{\Theta b}(E)\longrightarrow\mathcal{H}_{\Theta b}(E)$ is called a
\emph{convolution operator on }$\mathcal{H}_{\Theta b}(E)$ if it is
translation invariant, that is,
\[
L(\tau_{a}f)=\tau_{a}(L(f))
\]
for all $a\in E$ and $f\in\mathcal{H}_{\Theta b}(E).$\newline(c) For each
functional $T\in\lbrack\mathcal{H}_{\Theta b}(E)]^{\prime}$, the operator
$\bar{\Gamma}_{\Theta}(T)$ is defined by
\[
\bar{\Gamma}_{\Theta}(T)\colon\mathcal{H}_{\Theta b}(E)\longrightarrow
\mathcal{H}_{\Theta b}(E)~,~\bar{\Gamma}_{\Theta}(T)(f)=T\ast f,
\]
where the \textit{convolution product} $T\ast f$ is defined by
\[
\left(  T\ast f\right)  \left(  x\right)  =T\left(  \tau_{-x}f\right)
~for~every~x\in E.
\]
(d) $\delta_{0}\in\lbrack\mathcal{H}_{\Theta b}(E)]^{\prime}$ is the linear
functional defined by
\[
\delta_{0}\colon\mathcal{H}_{\Theta b}(E)\longrightarrow\mathbb{C}%
~,~\delta_{0}(f)=f(0).
\]
}} }
\end{definition}

\subsection{Hypercyclicity results}

%In this section we illustrate how to combine the results of the previous
%section with Theorem \ref{main1} and Theorem \ref{main2} to provide new examples of %hypercyclic
%convolution operators. We choose the class of Lorentz summing polynomials as a
%prototype of model.

Using the techniques developed in Section \ref{sec2}, in the final section we
shall provide new nontrivial applications of the following hypercyclicity results:

\begin{theorem}
\label{main1}\cite[Theorem 2.7]{BBFJ} Let $E^{\prime}$ be separable and
$(\mathcal{P}_{\Theta}(^{m}E))_{m=0}^{\infty}$ be a $\pi_{1}$-holomorphy type
from $E$ to $\mathbb{C}$. Then every convolution operator on $\mathcal{H}%
_{\Theta b}(E)$ which is not a scalar multiple of the identity is hypercyclic.
\end{theorem}

\begin{theorem}
\label{main2}\cite[Theorem 2.8]{BBFJ} Let $E^{\prime}$ be separable,
$(\mathcal{P}_{\Theta}(^{m}E))_{m=0}^{\infty}$ be a $\pi_{1}$-$\pi_{2}%
$-holomorphy type and $T\in\lbrack\mathcal{H}_{\Theta b}(E)]^{\prime}$ be a
linear functional which is not a scalar multiple of $\delta_{0}$. Then
$\bar{\Gamma}_{\Theta}(T)$ is a convolution operator that is not a scalar
multiple of the identity, hence hypercyclic.
\end{theorem}

\begin{remark}
\rm Since the proofs of Theorems \ref{main1} and \ref{main2} are based on
the hypercyclicity criterion obtained by Kitai \cite{kitai} and later on
rediscovered by Gethner and Shapiro \cite{gethner}, the convolution operators
of these theorems are in fact \emph{mixing}, a property stronger than
hypercyclicity. 
\end{remark}

\subsection{Existence and approximation results}

\begin{definition}
\textrm{\textrm{\label{definition_exp_space}\textrm{ Let $(\mathcal{P}%
_{\Theta}(^{m}E))_{m=0}^{\infty}$ be a $\pi_{1}$-holomorphy type from $E$ to
$\mathbb{C}$. An entire function $f\in\mathcal{H}(E^{\prime})$ is said to be
\textit{of $\Theta^{\prime}$-exponential type} if\newline(i) $\hat{d}%
^{m}f(0)\in\mathcal{P}_{\Theta^{\prime}}(^{m}E^{\prime})$ for every
$m\in\mathbb{N}_{0}$;\newline(ii) There are constants $C\geq0$ and $c>0$ such
that%
\[
\Vert{\hat{d}}^{m}f(0)\Vert_{\Theta^{\prime}}\leq Cc^{m},
\]
for all $m\in\mathbb{N}_{0}$. } } }

\textrm{\textrm{\textrm{The vector space of all such functions is denoted by
$Exp_{\Theta^{\prime}}(E^{\prime})$. } } }
\end{definition}

\begin{definition}
\textrm{\textrm{\cite[Definition 4.1]{favaro-jatoba1} \label{division space}%
\textrm{Let $U$ be an open subset of $E$ and $\mathcal{F}(U)$ a collection of
holomorphic functions from $U$ into $\mathbb{C}$. We say that $\mathcal{F}(U)$
is \emph{closed under division } if, for each $f$ and $g$ in $\mathcal{F}(U)$,
with $g\neq0$ and $h=f/g$ a holomorphic function on $U$, we have
$h\in\mathcal{F}(U)$.\newline The quotient notation $h=f/g$ means that
$f(x)=h(x)\cdot g(x)$, for all $x\in U$. } } }
\end{definition}

Now we are able to enunciate two results for convolution equations defined on
$\mathcal{H}_{\Theta b}(E)$ that we shall use, together with the techniques of
Section \ref{sec2}, to obtain new existence and approximation results for
convolution equations:

\begin{theorem}
\label{teoremaAproximacao1}\cite[Theorem 4.2]{favaro-jatoba1} If
$(\mathcal{P}_{\Theta}(^{m}E))_{m=0}^{\infty}$ is a $\pi_{1}$-$\pi_{2}
$-holomorphy type, $Exp_{\Theta^{\prime}}(E^{\prime})$ is closed under
division and $L\colon\mathcal{H}_{\Theta b}(E)\longrightarrow\mathcal{H}%
_{\Theta b}(E)$ is a convolution operator, then the vector subspace of
$\mathcal{H}_{\Theta b}(E)$ generated by the exponential polynomial solutions
of the homogeneous equation $L=0,$ is dense in the closed subspace of all
solutions of the homogeneous equation, that is, the vector subspace of
$\mathcal{H}_{\Theta b}(E)$ generated by
\[
\mathcal{L=}\left\{  P\exp\varphi;P\in\mathcal{P}_{\Theta}\left(
^{m}E\right)  ,m\in\mathbb{N}_{0},\varphi\in E^{\prime},L\left(  P\exp
\varphi\right)  =0\right\}
\]
is dense in
\[
\ker L=\left\{  f\in\mathcal{H}_{\Theta b}(E) ;L f=0\right\}  .
\]

\end{theorem}

\begin{theorem}
\label{Teorema de existencia}\cite[Theorem 4.4]{favaro-jatoba1} If
$(\mathcal{P}_{\Theta}(^{m}E))_{m=0}^{\infty}$ is a $\pi_{1}$-$\pi_{2}%
$-holomorphy type, $Exp_{\Theta^{\prime}}(E^{\prime})$ is closed under
division and $L\colon\mathcal{H}_{\Theta b}(E)\longrightarrow\mathcal{H}%
_{\Theta b}(E)$ is a non zero convolution operator, then $L$ is onto, that is,
$L\left(  \mathcal{H}_{\Theta b}\left(  E\right)  \right) =\mathcal{H}_{\Theta
b}\left(  E\right)  $.
\end{theorem}

\section{Applications\label{sec4}}

\subsection{Lorentz nuclear and summing polynomials: the basics}

For the sake of completeness we will recall the concepts of Lorentz summing
polynomials introduced in \cite{Matos-Pellegrino} and Lorentz nuclear
polynomials introduced in \cite{favaromatospellegrino} and related results. We
start introducing some notations.

We denote by $c_{0}(E)$ the Banach space (with the $\sup$ norm $\left\Vert
.\right\Vert _{\infty}$) composed by the sequences $(x_{j})_{j=1}^{\infty}$ in
the Banach space $E$ so that $\lim_{n\rightarrow\infty}x_{n}=0,$ and
$c_{00}(E)$ is the subspace of $c_{0}(E)$ formed by the sequences
$(x_{j})_{j=1}^{\infty}$ for which there is a $N_{0}$ such that $x_{n}=0$ for
all $n\geq N_{0}.$ When $E=\mathbb{K}:=\mathbb{R}$ or $\mathbb{C}$ we write
$c_{0}$ and $c_{00}$ instead of $c_{0}(\mathbb{K})$ and $c_{00}(\mathbb{K}),$
respectively. If $u=(u_{j})\in c_{00}(E)$, the symbol $card(u)$ denotes the
cardinality of the set $\{j;u_{j}\neq0\}.$

As usual $\ell_{\infty}(E)$ represents the Banach space of bounded sequences
in the Banach space $E$, with the $\sup$ norm and $\ell_{\infty}:=\ell
_{\infty}(\mathbb{K}).$ If $m\in\mathbb{N}$, $(x_{j})_{j=1}^{m}$ denotes
$(x_{1},\ldots,x_{m},0,0,\ldots),$ and when $(x_{j})_{j=1}^{\infty}$ is a
sequence of positive real numbers, we say that $(x_{j})_{j=1}^{\infty}$ admits
a non-increasing rearrangement if there is an injection $\pi:\mathbb{N}%
\rightarrow\mathbb{N}$ such that $x_{\pi(1)}\geq x_{\pi(2)}\geq\cdots$ and
$x_{j}=x_{\pi(i)}$ for some $i$ whenever $x_{j}\neq0.$ If $p\geq1,$ then
$p^{\prime}$ denotes the conjugate of $p$, i.e., $\frac{1}{p}+\frac
{1}{p^{\prime}}=1.$

%The space of continuous linear operators from $E$ to $F$ will be denoted by
%$\mathcal{L}(E;F)$ and the space of continuous $n$-homogeneous polynomials
%from $E$ to $F$ will be represented by $\mathcal{P}(^{n}E;F).$ In both
%classes, we consider the $\sup$ norm. We also represent the space of finite
%type polynomials from $E$ to $F$ by $\mathcal{P}_{f}(^{n}E;F).$ For the theory
%of polynomials and multilinear mappings in Banach spaces we refer to
%\cite{Dineen, Mujica}.

\begin{definition}
\textrm{\textrm{\cite{Matos-Pellegrino} Let $E$ be a Banach space and
$0<r,q<+\infty$. } }

\textrm{\textrm{(a) For $x=(x_{j})_{j=1}^{\infty}\in\ell_{\infty}(E)$ we
define
\[
a_{E,n}(x):=\inf\left\{  \left\Vert x-u\right\Vert _{\infty};u\in
c_{00}(E)\text{ and }card(u)<n\right\}  .
\]
} }

\textrm{\textrm{(b) The Lorentz sequence space $\ell_{(r,q)}(E)$ consists of
all sequences $x=(x_{j})_{j=1}^{\infty}\in\ell_{\infty}(E)$ such that%
\[
\left(  n^{\frac{1}{r}-\frac{1}{q}}a_{E,n}(x)\right)  _{n=1}^{\infty}\in
\ell_{q}.
\]
For $x\in\ell_{(r,q)}(E)$ we define the quasi-norm
\[
\left\Vert x\right\Vert _{(r,q)}=\left\Vert \left(  n^{\frac{1}{r}-\frac{1}%
{q}}a_{E,n}(x)\right)  _{n=1}^{\infty}\right\Vert _{q}
\]
} }

\textrm{\textrm{(c) $\ell_{(r,q)}^{w}(E)$ is the space of all sequences
$(x_{j})_{j=1}^{\infty}$ in $E$ such that $\left\Vert (\varphi(x_{n}%
))_{n=1}^{\infty}\right\Vert _{(r,q)}<\infty$ for every $\varphi\in E^{\prime
}.$ For $x\in\ell_{(r,q)}^{w}(E)$ we define the quasi-norm
\[
\left\Vert (x_{n})_{n=1}^{\infty}\right\Vert _{w,(r,q)}:=\sup_{\left\Vert
\varphi\right\Vert \leq1}\left\Vert (\varphi(x_{n}))_{n=1}^{\infty}\right\Vert
_{_{(r,q)}}.
\]
} }

\textrm{\textrm{If endowed with the respective quasi-norms, $\ell_{(r,q)}(E)$ and
$\ell_{(r,q)}^{w}(E)$ become complete spaces. } }
\end{definition}

It is well known that $\ell_{(r,q)}(E)\subset c_{0}(E)$ and if $x=(x_{j}%
)_{j=1}^{\infty}\in c_{0}(E),$ then the sequence $(\left\Vert x_{j}\right\Vert
)_{j=1}^{\infty}$ admits a non-increasing rearrangement.

\begin{definition}
\textrm{\textrm{\cite[Definition 4.1]{Matos-Pellegrino} \label{sb} If
$0<p,q,r,s<\infty,$ an $n$-homogeneous polynomial $P\in\mathcal{P}(^{n}E;F)$
is Lorentz\textbf{ }$((s,p);(r,q))$-summing if $(P(x_{j}))_{j=1}^{\infty}%
\in\ell_{(s,p)}(F)$ for each $(x_{j})_{j=1}^{\infty}\in\ell_{(r,q)}^{w}(E).$ }
}
\end{definition}

The vector space composed by the Lorentz $((s,p);(r,q))$-summing
$n$-homogeneous polynomials from $E$ to $F$ is denoted by $\mathcal{P}%
_{as((s,p);(r,q))}(^{n}E;F).$ When $n=1$ we write $\mathcal{L}%
_{as((s,p);(r,q))}(E;F).$

When $s=p$, we write $\mathcal{L}_{as(s;(r,q))}$ instead of $\mathcal{L}%
_{as((s,s);(r,q))}$; when $r=q,$ we denote $\mathcal{L}_{as((s,p);q)}$ instead
of $\mathcal{L}_{as((s,p);(q,q))}.$

Note that when $n=1$, $s=p$ and $r=q$ we have the usual concept of absolutely
$(p;q)$-summing operator. The space of absolutely $(p;q)$-summing operators
from $E$ to $F$ is represented by $\mathcal{L}_{as(p;q)}(E;F)$. When $p=q,$ we
simply write $\mathcal{L}_{as,p}$ instead of $\mathcal{L}_{as(p;q)}$. For the
theory of absolutely summing linear operators we refer to \cite{Diestel}.

\begin{theorem}
\cite[Theorem 4.2]{Matos-Pellegrino} \label{novoJ}For $P\in\mathcal{P}%
(^{n}E;F)$, the following conditions are equivalent:

(1) $P$ is Lorentz $((s,p);(r,q))$-summing.

(2) There is $C\geq0$ such that
\[
\left\Vert (P(x_{j}))_{j=1}^{m}\right\Vert _{(s,p)}\leq C\left\Vert
(x_{j})_{j=1}^{m}\right\Vert _{w,(r,q)}^{n}%
\]
for all $m\in\mathbb{N}$ and $x_{1},\dots,x_{m}\in E$.

(3) There is $C\geq0$ such that
\[
\left\Vert (P(x_{j}))_{j=1}^{\infty}\right\Vert _{(s,p)}\leq C\left\Vert
(x_{j})_{j=1}^{\infty}\right\Vert _{w,(r,q)}^{n}%
\]
for all $(x_{j})_{j=1}^{\infty}\in\ell_{(r,q)}^{w}(E)$.

The infimum of the constants $C$ for which the above inequalities hold is a
quasi-norm (denoted by $\left\Vert .\right\Vert _{as((s,p);(r,q))}$) for
$\mathcal{P}_{as((s,p);(r,q))}(^{n}E;F)$ and, under this quasi-norm,
$\mathcal{P}_{as((s,p);(r,q))}(^{n}E;F)$ is complete.
\end{theorem}

\begin{definition}
\textrm{\textrm{\cite[Definition 4.2]{favaromatospellegrino} \label{1.2}Let
$E$ and $F$ be Banach spaces, $n\in\mathbb{N}$ and $r,q,s,p\in\lbrack
1,\infty\lbrack$ such that $r\leq q$, $s^{\prime}\leq p^{\prime}$ and%
\[
1\leq\frac{1}{q}+\frac{n}{p^{\prime}}.
\]
An $n$-homogeneous polynomial $P:E\rightarrow F$ is Lorentz $((r,q);(s,p))$%
-nuclear if
\begin{equation}
P(x)=%
%TCIMACRO{\tsum \limits_{j=1}^{\infty}}%
%BeginExpansion
{\textstyle\sum\limits_{j=1}^{\infty}}
%EndExpansion
\lambda_{j}(\varphi_{j}(x))^{n}y_{j},\label{aaas}%
\end{equation}
with $(\lambda_{j})_{j=1}^{\infty}\in\ell_{(r,q)}$, $(\varphi_{j}%
)_{j=1}^{\infty}\in\ell_{(s^{\prime},p^{\prime})}^{w}(E^{\prime})$ and
$(y_{j})_{j=1}^{\infty}\in\ell_{\infty}(F).$ } }
\end{definition}

We denote by $\mathcal{P}_{N,((r,q);(s,p))}(^{n}E;F)$ the subset of
$\mathcal{P}(^{n}E;F)$ composed by the $n$-homogeneous polynomials which are
Lorentz $((r,q);(s,p))$-nuclear. We define%
\begin{equation}
\left\Vert P\right\Vert _{N,((r,q);(s,p))}=\inf\left\Vert (\lambda_{j}%
)_{j=1}^{\infty}\right\Vert _{(r,q)}\left\Vert (\varphi_{j})_{j=1}^{\infty
}\right\Vert _{w,(s^{\prime},p^{\prime})}^{n}\left\Vert (y_{j})_{j=1}^{\infty
}\right\Vert _{\infty},\nonumber
\end{equation}
where the infimum is considered for all representations of $P\in
\mathcal{P}_{N,((r,q);(s,p))}(^{n}E;F)$ of the form (\ref{aaas}).

Note that%
\begin{equation}
\left\Vert P\right\Vert \leq\left\Vert P\right\Vert _{N,((r,q);(s,p))}%
.\nonumber
\end{equation}

From now on, unless stated otherwise, $r,q\in]1,\infty\lbrack$ and
$s,p\in\lbrack1,\infty\lbrack,$ with $r\leq q$ and $s^{\prime}\leq p^{\prime
}.$

\begin{proposition}
\cite[Propositions 4.3 and 4.4]{favaromatospellegrino} \label{t1.4} The space
$\left(  \mathcal{P}_{N,((r,q);(s,p))}(^{n}E;F),\left\Vert \cdot\right\Vert
_{N,((r,q);(s,p))}\right)  $ is a complete quasi-normed space. Besides, for
$t_{n}$ given by
\[
\frac{1}{t_{n}}=\frac{1}{q}+\frac{n}{p^{\prime}},
\]
there is a $M\geq0$ so that%
\begin{equation}
\left\Vert P+Q\right\Vert _{N,((r,q);(s,p))}^{t_{n}}\leq M\left(  \left\Vert
P\right\Vert _{N,((r,q);(s,p))}^{t_{n}}+\left\Vert Q\right\Vert
_{N,((r,q);(s,p))}^{t_{n}}\right)  .\nonumber
\end{equation}
For this reason we call this quasi-norm a \textquotedblleft quasi- $t_{n}%
$-norm\textquotedblright.
\end{proposition}

\begin{theorem}
\cite[Theorem 5.4]{favaromatospellegrino} \label{555}If $E^{\prime}$ has the
bounded approximation property, then the linear mapping
\[
\Psi:\mathcal{P}_{N,((r,q);(s,p))}(^{n}E;F)^{\prime}\rightarrow\mathcal{P}%
_{as((r^{\prime},q^{\prime});(s^{\prime},p^{\prime}))}(^{n}E^{\prime
};F^{\prime})
\]
given by $\Psi(T)=P_{T}$ is a topological isomorphism, where the map
$P_{T}:E^{\prime}\rightarrow F^{\prime}$ is given by
\[
P_{T}(\varphi)(y)=T(\varphi^{n}y).
\]

\end{theorem}

\subsection{New hypercyclic, existence and approximation results}

Now, suppose that $E^{\prime}$ has the bounded approximation property. The
three steps below are common steps to obtain hypercyclic results (Theorems
\ref{main1} and \ref{main2}) and existence and approximation results (Theorems
\ref{teoremaAproximacao1} and \ref{Teorema de existencia}) for convolution operators:

(1) To obtain the spaces $\mathcal{P}_{\widetilde{N},\left(  \left(
r,q\right)  ;\left(  s,p\right)  \right)  }\left(  ^{n}E\right)  $, for all
$n\in\mathbb{N}$, according to Definition \ref{def555}.

(2) To prove that $\left(  \mathcal{P}_{\widetilde{N},\left(  \left(
r,q\right)  ;\left(  s,p\right)  \right)  }\left(  ^{n}E\right)  \right)
_{n=0}^{\infty}$ is a holomorphy type.

(3) To prove that $\left(  \mathcal{P}_{\widetilde{N},\left(  \left(
r,q\right)  ;\left(  s,p\right)  \right)  }\left(  ^{n}E\right)  \right)
_{n=0}^{\infty}$ is a $\pi_{1}$-$\pi_{2}$-holomorphy type.

\bigskip

A further step to obtain Theorems \ref{teoremaAproximacao1} and
\ref{Teorema de existencia} is:

\bigskip

(4) To prove that $Exp_{as\left(  \left(  r^{\prime},q^{\prime}\right)
;\left(  s^{\prime},p^{\prime}\right)  \right)  }\left(  E^{\prime}\right)  $
(see Definition \ref{definition_exp_space}(b)) is closed under division.

\bigskip

Step (1) is satisfied due to Proposition \ref{t1.4} and Theorem \ref{555}. In
fact, Theorem \ref{555} assures that the Borel transform is an isomorphism
between $\mathcal{P}_{N,((r,q);(s,p))}(^{n}E;F)^{\prime}$ and $\mathcal{P}%
_{as((r^{\prime},q^{\prime});(s^{\prime},p^{\prime}))}(^{n}E^{\prime
};F^{\prime}).$ Thus, we can consider, for each $n\in\mathbb{N}$, the space
$\mathcal{P}_{\widetilde{N},\left(  \left(  r,q\right)  ;\left(  s,p\right)
\right)  }\left(  ^{n}E\right)  $, of all \emph{Lorentz $((r,q);(s,p))$%
-quasi-nuclear $n$-homogeneous polynomials from $E$ to $\mathbb{C}$},
according to Definition \ref{def555}.

Now, let us prove Step (2), i.e., $\left(  \mathcal{P}_{\widetilde{N},\left(
\left(  r,q\right)  ;\left(  s,p\right)  \right)  }\left(  ^{n}E\right)
\right)  _{n=0}^{\infty}$ is a holomorphy type. We only have to prove that
$\left(  \mathcal{P}_{\widetilde{N},\left(  \left(  r,q\right)  ;\left(
s,p\right)  \right)  }\left(  ^{n}E\right)  \right)  _{n=0}^{\infty}$ is
stable for derivatives with constant $C_{n,k}=\frac{n!}{\left(  n-k\right)
!}$ (see Proposition \ref{proposicao d^kP} below) and the result follows from
Corollary \ref{corollary_holom_type}.

\begin{proposition}
\label{proposicao d^kP}If $P\in\mathcal{P}_{N,\left(  \left(  r,q\right)
;\left(  s,p\right)  \right)  }\left(  ^{n}E;F\right)  ,$ $k=1,\ldots,n$ and
$x\in E,$ then $\hat{d}^{k}P\left(  x\right)  \in\mathcal{P}_{N,\left(
\left(  r,q\right)  ;\left(  s,p\right)  \right)  }\left(  ^{k}E;F\right)  $
and%
\[
\left\Vert \hat{d}^{k}P\left(  x\right)  \right\Vert _{N,\left(  \left(
r,q\right)  ;\left(  s,p\right)  \right)  }\leq\frac{n!}{\left(  n-k\right)
!}\left\Vert P\right\Vert _{N,\left(  \left(  r,q\right)  ;\left(  s,p\right)
\right)  }\left\Vert x\right\Vert ^{n-k}.
\]

\end{proposition}

\begin{proof}
Let%
\[
P\left(  x\right)  =\sum\limits_{j=1}^{\infty}\lambda_{j}\left(  \varphi
_{j}\left(  x\right)  \right)  ^{n}y_{j},
\]
with $\left(  \lambda_{j}\right)  _{j=1}^{\infty}\in\ell_{\left(  r,q\right)
},$ $\left(  \varphi_{j}\right)  _{j=1}^{\infty}\in\ell_{\left(  s^{\prime
},p^{\prime}\right)  }^{w}\left(  E^{\prime}\right)  $ and $\left(
y_{j}\right)  _{j=1}^{\infty}\in\ell_{\infty}\left(  F\right)  .$ Then%
\begin{equation}
\hat{d}^{k}P\left(  x\right)  =\frac{n!}{\left(  n-k\right)  !}\sum
\limits_{j=1}^{\infty}\lambda_{j}\left(  \varphi_{j}\left(  x\right)  \right)
^{n-k}\varphi_{j}^{k}y_{j}, \label{representacao-d^kP}%
\end{equation}
for $k=1,\ldots,n.$ Let $y=\frac{x}{\left\Vert x\right\Vert },$ and note that
\[
\left\Vert \left(  \lambda_{j}\left(  \varphi_{j}\left(  y\right)  \right)
^{n-k}\right)  _{j=1}^{\infty}\right\Vert _{\left(  r,q\right)  }%
\leq\left\Vert \left(  \lambda_{j}\right)  _{j=1}^{\infty}\right\Vert
_{\left(  r,q\right)  }\sup_{j\in\mathbb{N}}\left\vert \varphi_{j}\left(
y\right)  \right\vert ^{n-k}.
\]

Now we obtain%
\begin{align*}
&  \frac{n!}{\left(  n-k\right)  !}\left\Vert x\right\Vert ^{n-k}\left\Vert
\left(  \lambda_{j}\left(  \varphi_{j}\left(  y\right)  \right)
^{n-k}\right)  _{j=1}^{\infty}\right\Vert _{\left(  r,q\right)  }\left\Vert
\left(  \varphi_{j}\right)  _{j=1}^{\infty}\right\Vert _{w,\left(  s^{\prime
},p^{\prime}\right)  }^{k}\left\Vert \left(  y_{j}\right)  _{j=1}^{\infty
}\right\Vert _{\infty}\\
&  \leq\frac{n!}{\left(  n-k\right)  !}\left\Vert x\right\Vert ^{n-k}%
\left\Vert \left(  \lambda_{j}\right)  _{j=1}^{\infty}\right\Vert _{\left(
r,q\right)  }\sup_{j\in\mathbb{N}}\left\vert \varphi_{j}\left(  y\right)
\right\vert ^{n-k}\left\Vert \left(  \varphi_{j}\right)  _{j=1}^{\infty
}\right\Vert _{w,\left(  s^{\prime},p^{\prime}\right)  }^{k}\left\Vert \left(
y_{j}\right)  _{j=1}^{\infty}\right\Vert _{\infty}\\
&  \leq\frac{n!}{\left(  n-k\right)  !}\left\Vert x\right\Vert ^{n-k}%
\left\Vert \left(  \lambda_{j}\right)  _{j=1}^{\infty}\right\Vert _{\left(
r,q\right)  }\left\Vert \left(  \varphi_{j}\right)  _{j=1}^{\infty}\right\Vert
_{w,\left(  s^{\prime},p^{\prime}\right)  }^{n-k}\left\Vert \left(
\varphi_{j}\right)  _{j=1}^{\infty}\right\Vert _{w,\left(  s^{\prime
},p^{\prime}\right)  }^{k}\left\Vert \left(  y_{j}\right)  _{j=1}^{\infty
}\right\Vert _{\infty}\\
&  =\frac{n!}{\left(  n-k\right)  !}\left\Vert x\right\Vert ^{n-k}\left\Vert
\left(  \lambda_{j}\right)  _{j=1}^{\infty}\right\Vert _{\left(  r,q\right)
}\left\Vert \left(  \varphi_{j}\right)  _{j=1}^{\infty}\right\Vert _{w,\left(
s^{\prime},p^{\prime}\right)  }^{n}\left\Vert \left(  y_{j}\right)
_{j=1}^{\infty}\right\Vert _{\infty}<+\infty.
\end{align*}
Thus $\left(  \ref{representacao-d^kP}\right)  $ is a valid Lorentz $\left(
\left(  r,q\right)  ;\left(  s,p\right)  \right)  $-nuclear representation of
$\hat{d}^{k}P\left(  x\right)  $ and in view of the last inequalities we can
write%
\[
\left\Vert \hat{d}^{k}P\left(  x\right)  \right\Vert _{N,\left(  \left(
r,q\right)  ;\left(  s,p\right)  \right)  }\leq\frac{n!}{\left(  n-k\right)
!}\left\Vert x\right\Vert ^{n-k}\left\Vert \left(  \lambda_{j}\right)
_{j=1}^{\infty}\right\Vert _{\left(  r,q\right)  }\left\Vert \left(
\varphi_{j}\right)  _{j=1}^{\infty}\right\Vert _{w,\left(  s^{\prime
},p^{\prime}\right)  }^{n}\left\Vert \left(  y_{j}\right)  _{j=1}^{\infty
}\right\Vert _{\infty}.
\]
Hence%
\[
\left\Vert \hat{d}^{k}P\left(  x\right)  \right\Vert _{N,\left(  \left(
r,q\right)  ;\left(  s,p\right)  \right)  }\leq\frac{n!}{\left(  n-k\right)
!}\left\Vert P\right\Vert _{N,\left(  \left(  r,q\right)  ;\left(  s,p\right)
\right)  }\left\Vert x\right\Vert ^{n-k}%
\]
as we wanted to show.
\end{proof}

Now we are able to define the space of Lorentz $\left(  \left(  r,q\right)
;\left(  s,p\right)  \right)  $-quasi-nuclear entire mappings of bounded type,
according to Definition \ref{Definition f holomorfia}.

\begin{definition}
\textrm{\textrm{An entire mapping $f\colon E\longrightarrow\mathbb{C}$ is said
to be \emph{Lorentz }$\left(  \left(  r,q\right)  ;\left(  s,p\right)
\right)  $\emph{-quasi-nuclear of bounded type} if\newline(1) $\hat{d}^{n}%
$\ $f(0)\in\mathcal{P}_{\tilde{N},\left(  \left(  r,q\right)  ;\left(
s,p\right)  \right)  }\left(  ^{n}E\right)  ,$ for all $n\in\mathbb{N}_{0}%
,$\newline(2) $\lim\limits_{n\rightarrow\infty}\left(  \frac{1}{n!}\left\Vert
\hat{d}^{n}f(0)\right\Vert _{\tilde{N},\left(  \left(  r,q\right)  ;\left(
s,p\right)  \right)  }\right)  ^{\frac{1}{n}}=0.$\newline The space of all
entire mappings $f\colon E\longrightarrow\mathbb{C}$ that are Lorentz $\left(
\left(  r,q\right)  ;\left(  s,p\right)  \right)  $\emph{-}quasi-nuclear of
bounded type is denoted by $\mathcal{H}_{\tilde{N}b,\left(  \left(
r,q\right)  ;\left(  s,p\right)  \right)  }(E)$ and it is a Fr\'{e}chet space
with the topology generated by the family of seminorms:
\begin{equation}
f\in\mathcal{H}_{\widetilde{N}b,\left(  \left(  r,q\right)  ;\left(
s,p\right)  \right)  }(E)\mapsto\Vert f\Vert_{\widetilde{N}b,\left(  \left(
r,q\right)  ;\left(  s,p\right)  \right)  ,\rho}=\sum_{m=0}^{\infty}\frac
{\rho^{m}}{m!}\Vert\hat{d}^{m}f(0)\Vert_{\widetilde{N},\left(  \left(
r,q\right)  ;\left(  s,p\right)  \right)  },\label{seminorma_f}%
\end{equation}
for all $\rho>0$. } }
\end{definition}

%\begin{corollary}
%The sequence $\left(  \mathcal{P}_{\widetilde{N},\left(  \left(  r,q\right)
%;\left(  s,p\right)  \right)  }\left(  ^{n}E\right)  \right)  _{n=0}^{\infty}$
%is a holomorphy type.
%\end{corollary}
%\begin{proof}
%It is enough to use Proposition \ref{proposicao d^kP} together with Corollary
%\ref{corollary_holom_type}.
%\end{proof}

Now we have to prove that $\left(  \mathcal{P}_{\widetilde{N},\left(  \left(
r,q\right)  ;\left(  s,p\right)  \right)  }\left(  ^{n}E\right)  \right)
_{n=0}^{\infty}$ is a $\pi_{1}$-$\pi_{2}$-holomorphy type.

%We use Proposition \ref{lema_para_pi1} to prove that$\left(  \mathcal{P}%
%_{\widetilde{N},\left(  \left(  r,q\right)  ;\left(  s,p\right)  \right)
%}\left(  ^{n}E\right)  \right)  _{n=0}^{\infty}$ is a $\pi_{1}$-holomorphy type.

\begin{proposition}
\label{pi1_tipo_holomorfia}$\left(  \mathcal{P}_{\widetilde{N},\left(  \left(
r,q\right)  ;\left(  s,p\right)  \right)  }\left(  ^{n}E\right)  \right)
_{n=0}^{\infty}$ is a $\pi_{1}$-holomorphy type.
\end{proposition}

\begin{proof}
In \cite[Example 4.5]{favaromatospellegrino} it was proved that $\mathcal{P}%
_{f}\left(  ^{n}E\right)  $ is contained in $\mathcal{P}_{N,\left(  \left(
r,q\right)  ;\left(  s,p\right)  \right)  }\left(  ^{n}E\right)  $ and
\[
\Vert\phi^{n}\Vert_{N,\left(  \left(  r,q\right)  ;\left(  s,p\right)
\right)  }=\Vert\phi\Vert^{n}%
\]
for all $\phi\in E^{\prime}$ and $n\in\mathbb{N}$. Besides, from \cite[Lemma
4.10]{favaromatospellegrino} we know that the space of finite type polynomials
$\ \mathcal{P}_{f}\left(  ^{n}E\right)  $ is dense in $\left(  \mathcal{P}%
_{N,((r,q);(s,p))}(^{n}E),\left\Vert \cdot\right\Vert _{N,((r,q);(s,p))}%
\right)  $. Thus, it follows from Proposition \ref{lema_para_pi1} that
$\mathcal{P}_{f}\left(  ^{n}E\right)  $ is dense in $\left(  \mathcal{P}%
_{\widetilde{N},((r,q);(s,p))}(^{n}E),\left\Vert \cdot\right\Vert
_{\widetilde{N},((r,q);(s,p))}\right)  $ and
\begin{equation}
\Vert\phi^{n}\Vert_{\widetilde{N},\left(  \left(  r,q\right)  ;\left(
s,p\right)  \right)  }=\Vert\phi\Vert^{n}, \label{norma_phi}%
\end{equation}
for all $\phi\in E^{\prime}$ and $n\in\mathbb{N}$, since the constants $K$ and
$C_{\Delta_{n}}$, in this case, may be taken equal to $1$.
\end{proof}

\begin{proposition}
$\left(  \mathcal{P}_{\widetilde{N},\left(  \left(  r,q\right)  ;\left(
s,p\right)  \right)  }\left(  ^{n}E\right)  \right)  _{n=0}^{\infty}$ is a
$\pi_{2}$-holomorphy type.
\end{proposition}

\begin{proof}
Let $T\in\left[  \mathcal{H}_{\widetilde{N}b,((r,q);(s,p))}(E)\right]
^{\prime}$, $n,k\in\mathbb{N}_{0},$ $k\leq n$ and $C,\rho>0$ be constants such
that
\begin{equation}
\left\vert T\left(  f\right)  \right\vert \leq C\left\Vert f\right\Vert
_{\widetilde{N}b,((r,q);(s,p)),\rho}~\text{for~every}~f\in\mathcal{H}%
_{\widetilde{N}b,((r,q);(s,p))}(E). \label{desig_norma_f}%
\end{equation}
For\textrm{ }$P\in\mathcal{P}_{\widetilde{N},((r,q);(s,p))}(^{n}E)$, we will
show that the $\left(  n-k\right)  $-homogeneous polynomial%
\begin{align*}
T\left(  \widehat{A(\cdot)^{k}}\right)  \colon E  &  \longrightarrow
\mathbb{C}\\
y  &  \mapsto T\left(  A(\cdot)^{k}y^{n-k}\right)
\end{align*}
where $A\colon E^{n}\longrightarrow C$ is the unique continuous symmetric
$n$-linear mapping such that $P=\hat{A},$ belongs to $\mathcal{P}%
_{\widetilde{N},((r,q);(s,p))}(^{n-k}E)$ and%
\[
\Vert T(\widehat{A(\cdot)^{k}})\Vert_{\widetilde{N},((r,q);(s,p))}\leq
C\cdot\rho^{k}\Vert P\Vert_{\widetilde{N},((r,q);(s,p))}.
\]
First, suppose that $P\in\mathcal{P}_{N,\left(  \left(  r,q\right)  ;\left(
s,p\right)  \right)  }(^{n}E)$. Thus
\[
P=%
%TCIMACRO{\dsum \limits_{j=1}^{\infty}}%
%BeginExpansion
{\displaystyle\sum\limits_{j=1}^{\infty}}
%EndExpansion
\lambda_{j}\varphi_{j}^{n}\text{,}%
\]
with $(\lambda_{j})\in\ell_{(r,q)}$ and $(\varphi_{j})\in\ell_{(s^{\prime
},p^{\prime})}^{w}(E^{\prime})$, and for every $y\in E$ we have%
\begin{align*}
T\left(  \widehat{A(\cdot)^{k}}\right)  (y)  &  =T(A(\cdot)^{k}y^{n-k}%
)=T\left(
%TCIMACRO{\dsum \limits_{j=1}^{\infty}}%
%BeginExpansion
{\displaystyle\sum\limits_{j=1}^{\infty}}
%EndExpansion
\lambda_{j}\varphi_{j}^{k}\varphi_{j}(y)^{n-k}\right) \\
&  =%
%TCIMACRO{\dsum \limits_{j=1}^{\infty}}%
%BeginExpansion
{\displaystyle\sum\limits_{j=1}^{\infty}}
%EndExpansion
\lambda_{j}T\left(  \varphi_{j}^{k}\right)  \varphi_{j}(y)^{n-k}.
\end{align*}
Now, to prove that $T\left(  \widehat{A(\cdot)^{k}}\right)  \in\mathcal{P}%
_{N,\left(  \left(  r,q\right)  ;\left(  s,p\right)  \right)  }(^{n-k}E),\ $we
only have to show that $(\lambda_{j}T\left(  \varphi_{j}^{k}\right)  )\in
\ell_{(r,q)}$ since we already have $(\varphi_{j})\in\ell_{(s^{\prime
},p^{\prime})}^{w}(E^{\prime}).$ First, note that%
\begin{gather*}
\left\Vert \left(  T\left(  \varphi_{j}^{k}\right)  \right)  _{j=1}^{\infty
}\right\Vert _{\infty}=\sup_{j}\left\vert T\left(  \varphi_{j}^{k}\right)
\right\vert \overset{(\ref{desig_norma_f})}{\leq}\sup_{j}\left\Vert
\varphi_{j}^{k}\right\Vert _{\widetilde{N},\left(  \left(  r,q\right)
;\left(  s,p\right)  \right)  ,\rho}\\
\overset{\left(  \ref{seminorma_f}\right)  }{=}C\rho^{k}\sup_{j}\Vert
\varphi_{j}^{k}\Vert_{\widetilde{N},\left(  \left(  r,q\right)  ;\left(
s,p\right)  \right)  }\overset{\left(  \ref{norma_phi}\right)  }{=}C\rho
^{k}\sup_{j}\Vert\varphi_{j}\Vert^{k}\\
=C\rho^{k}\left\Vert \left(  \Vert\varphi_{j}\Vert\right)  _{j=1}^{\infty
}\right\Vert _{\infty}^{k}\leq C\rho^{k}\left\Vert (\varphi_{j})_{j=1}%
^{\infty}\right\Vert _{w,(s^{\prime},p^{\prime})}^{k}.
\end{gather*}

Hence%

\begin{align*}
\left\Vert \left(  \lambda_{j}T\left(  \varphi_{j}^{k}\right)  \right)
_{j=1}^{\infty}\right\Vert _{(r,q)}  &  \leq\left\Vert \left(  \lambda
_{j}\right)  _{j=1}^{\infty}\right\Vert _{(r,q)}\left\Vert \left(  T\left(
\varphi_{j}^{k}\right)  \right)  _{j=1}^{\infty}\right\Vert _{\infty}\\
&  \leq C\rho^{k}\left\Vert \left(  \lambda_{j}\right)  _{j=1}^{\infty
}\right\Vert _{(r,q)}\left\Vert (\varphi_{j})_{j=1}^{\infty}\right\Vert
_{w,(s^{\prime},p^{\prime})}^{k}<\infty
\end{align*}
and it follows that $T\left(  \widehat{A(\cdot)^{k}}\right)  \in
\mathcal{P}_{N,\left(  \left(  r,q\right)  ;\left(  s,p\right)  \right)
}(^{n-k}E).$ Moreover%
\begin{align*}
\left\Vert T\left(  \widehat{A(\cdot)^{k}}\right)  \right\Vert _{N,\left(
\left(  r,q\right)  ;\left(  s,p\right)  \right)  }  &  \leq\left\Vert \left(
\lambda_{j}T\left(  \varphi_{j}^{k}\right)  \right)  _{j=1}^{\infty
}\right\Vert _{(r,q)}\left\Vert (\varphi_{j})_{j=1}^{\infty}\right\Vert
_{w,(s^{\prime},p^{\prime})}^{n-k}\\
&  \leq\left[  C\rho^{k}\left\Vert \left(  \lambda_{j}\right)  _{j=1}^{\infty
}\right\Vert _{(r,q)}\left\Vert (\varphi_{j})_{j=1}^{\infty}\right\Vert
_{w,(s^{\prime},p^{\prime})}^{k}\right]  \left\Vert (\varphi_{j}%
)_{j=1}^{\infty}\right\Vert _{w,(s^{\prime},p^{\prime})}^{n-k}\\
&  =C\rho^{k}\left\Vert \left(  \lambda_{j}\right)  _{j=1}^{\infty}\right\Vert
_{(r,q)}\left\Vert (\varphi_{j})_{j=1}^{\infty}\right\Vert _{w,(s^{\prime
},p^{\prime})}^{n}.
\end{align*}
and so%
\[
\left\Vert T\left(  \widehat{A(\cdot)^{k}}\right)  \right\Vert _{\widetilde
{N},\left(  \left(  r,q\right)  ;\left(  s,p\right)  \right)  }\overset
{\left(  \ref{dggg}\right)  }{\leq}\left\Vert T\left(  \widehat{A(\cdot)^{k}%
}\right)  \right\Vert _{N,\left(  \left(  r,q\right)  ;\left(  s,p\right)
\right)  }\leq C\rho^{k}\left\Vert P\right\Vert _{N,\left(  \left(
r,q\right)  ;\left(  s,p\right)  \right)  }.
\]
\newline Now, if we define%
\begin{gather*}
\psi:\mathcal{P}_{N,\left(  \left(  r,q\right)  ;\left(  s,p\right)  \right)
}(^{n}E)\rightarrow\mathcal{P}_{N,\left(  \left(  r,q\right)  ;\left(
s,p\right)  \right)  }(^{n-k}E)\\
\text{ \ \ \ \ \ \ \ \ \ \ \ }P\mapsto T\left(  \widehat{A(\cdot)^{k}}\right)
\end{gather*}
and proceeding in a similar way as in the proof of Theorem
\ref{Theo_stability_for_derivatives}, the result follows.
\end{proof}

Now we can state the hypercyclicity results in this new framework:

\begin{theorem}
Let $E^{\prime}$ be separable. Then every convolution operator on
$\mathcal{H}_{\widetilde{N}b,\left(  \left(  r,q\right)  ;\left(  s,p\right)
\right)  }\left(  E\right)  $ which is not a scalar multiple of the identity
is hypercyclic.
\end{theorem}

\begin{theorem}
Let $E^{\prime}$ be separable and $T\in\lbrack\mathcal{H}_{\widetilde
{N}b,\left(  \left(  r,q\right)  ;\left(  s,p\right)  \right)  }\left(
E\right)  ]^{\prime}$ be a linear functional which is not a scalar multiple of
$\delta_{0}$. Then $\bar{\Gamma}_{N,\left(  \left(  r,q\right)  ;\left(
s,p\right)  \right)  }(T)$ is a convolution operator that is not a scalar
multiple of the identity, hence hypercyclic.
\end{theorem}

\begin{remark}
\rm A recent result of Pinasco, Muro, Savransky \cite{MPS} shows that
nontrivial convolution operators on certain spaces of entire functions on a
Banach space are strongly mixing in the gaussian sense, in particular
frequently hypercyclic. We believe that, using the corresponding auxiliary
results from \cite{BBFJ}, their result holds for $\mathcal{H}_{\Theta b}(E)$,
when $\Theta$ is a $\pi_{1}$-holomorphy type. In this case, the reasoning used
to prove the two theorems above actually proves that the convolution operators
of the two theorems above can be proved to be strongly mixing in the gaussian
sense. 
\end{remark}

Now, according to Definition \ref{definition_exp_space}(b) we introduce the
\emph{Lorentz summing functions of exponential type}:

\begin{definition}
\textrm{\textrm{An entire mapping $f\colon E\longrightarrow\mathbb{C}$ is said
to be of \emph{Lorentz }$\left(  \left(  s,p\right)  ;\left(  r,q\right)
\right)  $\emph{-summing exponential type} if $\hat{d}^{n}$\ $f(0)\in
\mathcal{P}_{as\left(  \left(  s,p\right)  ;\left(  r,q\right)  \right)
}\left(  ^{n}E\right)  ,$ for all $n\in\mathbb{N}_{0},$ and there are $C\geq0$
and $c>0$ such that%
\[
\left\Vert \hat{d}^{n}f(0)\right\Vert _{as\left(  \left(  s,p\right)  ;\left(
r,q\right)  \right)  }\leq Cc^{n},
\]
for all $n\in\mathbb{N}_{0}.$ } }

\textrm{\textrm{The vector space of all these mappings is denoted by
$Exp_{as\left(  \left(  s,p\right)  ;\left(  r,q\right)  \right)  }\left(
E\right)  .$ } }
\end{definition}

This definition was motivated by the definition of mappings of exponential
type (see \cite{G}):

\begin{definition}
\textrm{\textrm{\label{def_exp_type} An entire mapping $f\colon
E\longrightarrow F$ is said to be of \emph{exponential type} if one of the
following equivalent conditions holds: } }

\textrm{\textrm{(i) There are $C\geq0$ and $c>0$ such that $\|f(x)\|\leq
C\exp{(c\|x\|)},$ for all $x\in E$. } }

\textrm{\textrm{(ii) There are $D\geq0$ and $d>0$ such that $\|\hat{d}%
^{m}f(0)\|\leq Cc^{m},$ for all $m\in\mathbb{N}$. } }

\textrm{\textrm{(iii) $\limsup_{m\rightarrow\infty} \|\hat{d}^{m}%
f(0)\|^{\frac{1}{m}}<+\infty$. } }

\textrm{\textrm{We denote by $Exp(E;F)$ the vector space of all entire
mappings of exponential type from $E$ into $F$. When $F$ is the scalar field
$\mathbb{C}$ we denote $Exp(E)$ instead of $Exp(E;\mathbb{C}).$ } }
\end{definition}

The next two lemmata will be necessary for the proof of the division theorem
(Theorem \ref{teodiv}). The first is a division result due to Gupta (see
\cite{G, Gupta}) .

\begin{lemma}
If $f$, $g$ and $h$ are entire mappings on $E$ with values in $\mathbb{C}$,
$f\neq0$, $h(x) = f(x)g(x)$ for all $x\in E$, with $f$ and $h$ of exponential
type on $E$, then $g$ is of exponential type on $E$.
\end{lemma}

%We need the following result to prove that $Exp_{as\left(  \left(  r^{\prime
%},q^{\prime}\right)  ;\left(  s^{\prime},p^{\prime}\right)  \right)  }\left(
%E\right)  $ is closed under division.

\begin{lemma}
\label{prod_funcoes}Let $r,q,s,p\in\left[  1,+\infty\right[  $ with $r\leq q$
and $F$ be a Banach space. If $f\in Exp\left(  F\right)  $ and $g\in
Exp_{as\left(  \left(  r^{\prime},q^{\prime}\right)  ;\left(  s^{\prime
},p^{\prime}\right)  \right)  }\left(  F\right)  $ then $fg$ is in
$Exp_{as\left(  \left(  r^{\prime},q^{\prime}\right)  ;\left(  s^{\prime
},p^{\prime}\right)  \right)  }\left(  F\right)  .$
\end{lemma}

\begin{proof}
For each $k\in\mathbb{N}$, it follows from the uniqueness of the power series
of a holomorphic function around a point of its domain that%
\[
\hat{d}^{k}\left(  fg\right)  (0)(x)=%
%TCIMACRO{\dsum \limits_{l=0}^{k}}%
%BeginExpansion
{\displaystyle\sum\limits_{l=0}^{k}}
%EndExpansion
k!\frac{1}{l!}\hat{d}^{l}f(0)(x)\frac{1}{(k-l)!}\hat{d}^{k-l}g(0)(x)
\]
for all $x\in F.$ Since $f\in Exp\left(  F\right)  $, there are $C\geq0$ and
$c>0$ such that%
\[
\left\Vert \hat{d}^{n}f(0)\right\Vert \leq Cc^{n},
\]
for every $n\in\mathbb{N}.$ For $\left\Vert (x_{j})_{j=1}^{\infty}\right\Vert
_{w,(s^{\prime},p^{\prime})}\leq1$, we have $\left\Vert x_{j}\right\Vert
\leq1$ for every $j\in\mathbb{N}$ and so%
\[
\left\vert \hat{d}^{n}f(0)(x_{j})\right\vert \leq Cc^{n},
\]
for every $n\in\mathbb{N}.$ Thus%

\[
\left\vert \hat{d}^{k}\left(  fg\right)  (0)(x_{j})\right\vert \leq C%
%TCIMACRO{\dsum \limits_{l=0}^{k}}%
%BeginExpansion
{\displaystyle\sum\limits_{l=0}^{k}}
%EndExpansion
\frac{k!}{l!(k-l)!}c^{l}\left\vert \hat{d}^{k-l}g(0)(x_{j})\right\vert ,
\]
for all $j\in\mathbb{N}$. Let $\pi\colon\mathbb{N}\longrightarrow\mathbb{N}$
be an injection. Since $r^{\prime}\geq q^{\prime}$ we have (using Lemma 3.1 of
\cite{Matos-Pellegrino})%
\begin{align*}
\left[
%TCIMACRO{\dsum \limits_{j=1}^{\infty}}%
%BeginExpansion
{\displaystyle\sum\limits_{j=1}^{\infty}}
%EndExpansion
\left(  j^{\frac{1}{r^{\prime}}-\frac{1}{q^{\prime}}}\left\vert \hat{d}%
^{k}\left(  fg\right)  (0)(x_{\pi(j)})\right\vert \right)  ^{q^{\prime}%
}\right]  ^{\frac{1}{q^{\prime}}}  &  \leq C%
%TCIMACRO{\dsum \limits_{l=0}^{k}}%
%BeginExpansion
{\displaystyle\sum\limits_{l=0}^{k}}
%EndExpansion
\frac{k!}{l!(k-l)!}c^{l}\left[
%TCIMACRO{\dsum \limits_{j=1}^{\infty}}%
%BeginExpansion
{\displaystyle\sum\limits_{j=1}^{\infty}}
%EndExpansion
\left(  j^{\frac{1}{r^{\prime}}-\frac{1}{q^{\prime}}}\left\vert \hat{d}%
^{k-l}g(0)(x_{\pi(j)})\right\vert \right)  ^{q^{\prime}}\right]  ^{\frac
{1}{q^{\prime}}}\\
&  \leq C%
%TCIMACRO{\dsum \limits_{l=0}^{k}}%
%BeginExpansion
{\displaystyle\sum\limits_{l=0}^{k}}
%EndExpansion
\frac{k!}{l!(k-l)!}c^{l}\left\Vert \left(  \hat{d}^{k-l}g(0)(x_{j})\right)
_{j=1}^{\infty}\right\Vert _{(r^{\prime},q^{\prime})}.
\end{align*}

By Theorem \ref{novoJ} we have%
\[
\left\Vert \left(  \hat{d}^{k-l}g(0)(x_{j})\right)  _{j=1}^{\infty}\right\Vert
_{(r^{\prime},q^{\prime})}\leq\left\Vert \hat{d}^{k-l}g(0)\right\Vert
_{as\left(  (r^{\prime},q^{\prime}),(s^{\prime},p^{\prime})\right)
}\left\Vert (x_{j})_{j=1}^{\infty}\right\Vert _{w,(s^{\prime},p^{\prime}%
)}^{k-l}\leq\left\Vert \hat{d}^{k-l}g(0)\right\Vert _{as\left(  (r^{\prime
},q^{\prime}),(s^{\prime},p^{\prime})\right)  },
\]
for all $l=1,\ldots,k.$ Thus%
\[
\left[
%TCIMACRO{\dsum \limits_{j=1}^{\infty}}%
%BeginExpansion
{\displaystyle\sum\limits_{j=1}^{\infty}}
%EndExpansion
\left(  j^{\frac{1}{r^{\prime}}-\frac{1}{q^{\prime}}}\left\vert \hat{d}%
^{k}\left(  fg\right)  (0)(x_{\pi(j)})\right\vert \right)  ^{q^{\prime}%
}\right]  ^{\frac{1}{q^{\prime}}}\leq C%
%TCIMACRO{\dsum \limits_{l=0}^{k}}%
%BeginExpansion
{\displaystyle\sum\limits_{l=0}^{k}}
%EndExpansion
\frac{k!}{l!(k-l)!}c^{l}\left\Vert \hat{d}^{k-l}g(0)\right\Vert _{as\left(
(r^{\prime},q^{\prime}),(s^{\prime},p^{\prime})\right)  }%
\]
and we conclude that $\left(  \left\vert \hat{d}^{k}\left(  fg\right)
(0)(x_{\pi(j)})\right\vert \right)  _{j=1}^{\infty}\in\ell_{(r^{\prime
},q^{\prime})}$ and thus $\left(  \hat{d}^{k}\left(  fg\right)  (0)(x_{\pi
(j)})\right)  _{j=1}^{\infty}\in c_{0}.$

Since $\pi$ is arbitrary we have%
\[
\left[
%TCIMACRO{\dsum \limits_{j=1}^{\infty}}%
%BeginExpansion
{\displaystyle\sum\limits_{j=1}^{\infty}}
%EndExpansion
\left(  j^{\frac{1}{r^{\prime}}-\frac{1}{q^{\prime}}}\left\vert \hat{d}%
^{k}\left(  fg\right)  (0)(x_{\sigma(j)})\right\vert \right)  ^{q^{\prime}%
}\right]  ^{\frac{1}{q^{\prime}}}=\left\Vert \left(  \hat{d}^{k}\left(
fg\right)  (0)(x_{j})\right)  _{j=1}^{\infty}\right\Vert _{(r^{\prime
},q^{\prime})}%
\]
for some injection $\sigma\colon\mathbb{N}\longrightarrow\mathbb{N}$. Hence%
\[
\left\Vert \left(  \hat{d}^{k}\left(  fg\right)  (0)(x_{j})\right)
_{j=1}^{\infty}\right\Vert _{(r^{\prime},q^{\prime})}\leq C%
%TCIMACRO{\dsum \limits_{l=0}^{k}}%
%BeginExpansion
{\displaystyle\sum\limits_{l=0}^{k}}
%EndExpansion
\frac{k!}{l!(k-l)!}c^{l}\left\Vert \hat{d}^{k-l}g(0)\right\Vert _{as\left(
(r^{\prime},q^{\prime}),(s^{\prime},p^{\prime})\right)  }.
\]
Now, let $0\neq(x_{j})_{j=1}^{\infty}\in\ell_{\left(  s^{\prime},p^{\prime
}\right)  }^{w}(F).$ Defining
\[
y_{j}=\frac{x_{j}}{\left\Vert (x_{j})_{j=1}^{\infty}\right\Vert _{w,(s^{\prime
},p^{\prime})}},
\]
for all $j\in\mathbb{N},$ we have
\[
\left\Vert (y_{j})_{j=1}^{\infty}\right\Vert _{w,(s^{\prime},p^{\prime})}=1
\]
and using the previous estimates we obtain%
\[
\left\Vert \left(  \hat{d}^{k}\left(  fg\right)  (0)(x_{j})\right)
_{j=1}^{\infty}\right\Vert _{(r^{\prime},q^{\prime})}\leq\left(  C%
%TCIMACRO{\dsum \limits_{l=0}^{k}}%
%BeginExpansion
{\displaystyle\sum\limits_{l=0}^{k}}
%EndExpansion
\frac{k!}{l!(k-l)!}c^{l}\left\Vert \hat{d}^{k-l}g(0)\right\Vert _{as\left(
(r^{\prime},q^{\prime}),(s^{\prime},p^{\prime})\right)  }\right)  \left\Vert
(x_{j})_{j=1}^{\infty}\right\Vert _{w,(s^{\prime},p^{\prime})}^{k}.
\]
Using again Theorem \ref{novoJ}, it follows that $\hat{d}^{k}g(0)\in
\mathcal{P}_{as\left(  \left(  r^{\prime},q^{\prime}\right)  ;\left(
s^{\prime},p^{\prime}\right)  \right)  }\left(  ^{k}F\right)  $ and%
\[
\left\Vert \hat{d}^{k}\left(  fg\right)  (0)\right\Vert _{as\left(
(r^{\prime},q^{\prime}),(s^{\prime},p^{\prime})\right)  }\leq C%
%TCIMACRO{\dsum \limits_{l=0}^{k}}%
%BeginExpansion
{\displaystyle\sum\limits_{l=0}^{k}}
%EndExpansion
\frac{k!}{l!(k-l)!}c^{l}\left\Vert \hat{d}^{k-l}g(0)\right\Vert _{as\left(
(r^{\prime},q^{\prime}),(s^{\prime},p^{\prime})\right)  }.
\]
Since $g\in Exp_{as\left(  \left(  r^{\prime},q^{\prime}\right)  ;\left(
s^{\prime},p^{\prime}\right)  \right)  }\left(  F\right)  $ there are $D\geq0$
and $d>0$ such that%
\[
\left\Vert \hat{d}^{n}g(0)\right\Vert _{as\left(  (r^{\prime},q^{\prime
}),(s^{\prime},p^{\prime})\right)  }\leq Dd^{n},
\]
for all $n\in\mathbb{N}$. Hence%
\[
\left\Vert \hat{d}^{k}\left(  fg\right)  (0)\right\Vert _{as\left(
(r^{\prime},q^{\prime}),(s^{\prime},p^{\prime})\right)  }\leq CD%
%TCIMACRO{\dsum \limits_{l=0}^{k}}%
%BeginExpansion
{\displaystyle\sum\limits_{l=0}^{k}}
%EndExpansion
\frac{k!}{l!(k-l)!}c^{l}d^{k-l}=CD\left(  c+d\right)  ^{k}%
\]
and $fg\in Exp_{as\left(  \left(  r^{\prime},q^{\prime}\right)  ;\left(
s^{\prime},p^{\prime}\right)  \right)  }\left(  F\right)  .$
\end{proof}

\begin{theorem}
[Division Theorem]\label{teodiv}Let $r,q,s,p\in\left[  1,+\infty\right[  $
with $r\leq q$ and $F$ be a Banach space. Then $Exp_{as\left(  \left(
r^{\prime},q^{\prime}\right)  ;\left(  s^{\prime},p^{\prime}\right)  \right)
}\left(  F\right)  $ is closed under division, that is, if $f,g$ and $h$ are
entire mappings on $F$ with values in $\mathbb{C}$, $f\neq0,$ and $h=fg$ with
$f$ and $h$ in $Exp_{as\left(  \left(  r^{\prime},q^{\prime}\right)  ;\left(
s^{\prime},p^{\prime}\right)  \right)  }\left(  F\right)  ,$ then $g$ is also
in $Exp_{as\left(  \left(  r^{\prime},q^{\prime}\right)  ;\left(  s^{\prime
},p^{\prime}\right)  \right)  }\left(  F\right)  .$
\end{theorem}

\begin{proof}
For each $k\in\mathbb{N}$ we have%
\[
\hat{d}^{k}h(0)(x)=f(0)\hat{d}^{k}g(0)(x)+%
%TCIMACRO{\dsum \limits_{l=1}^{k}}%
%BeginExpansion
{\displaystyle\sum\limits_{l=1}^{k}}
%EndExpansion
\frac{k!}{l!(k-l)!}\hat{d}^{l}f(0)(x)\hat{d}^{k-l}g(0)(x)
\]
and so%
\[
f(0)\hat{d}^{k}g(0)(x)=\hat{d}^{k}h(0)(x)-%
%TCIMACRO{\dsum \limits_{l=1}^{k}}%
%BeginExpansion
{\displaystyle\sum\limits_{l=1}^{k}}
%EndExpansion
\frac{k!}{l!(k-l)!}\hat{d}^{l}f(0)(x)\hat{d}^{k-l}g(0)(x)
\]
for all $x\in F.$ Let us suppose $f(0)\neq0.$ For $\left\Vert (x_{j}%
)_{j=1}^{\infty}\right\Vert _{w,(s^{\prime},p^{\prime})}\leq1$, we have
$\left\Vert x_{j}\right\Vert \leq1$ for every $j\in\mathbb{N}$. Since
$\Vert\cdot\Vert\leq\Vert\cdot\Vert_{as\left(  \left(  r^{\prime},q^{\prime
}\right)  ;\left(  s^{\prime},p^{\prime}\right)  \right)  },$ it follows from
Definition \ref{def_exp_type}(ii) that $g$ is of exponential type. So there
are $C\geq0$ and $c>0$ such that%
\[
\left\Vert \hat{d}^{k}g(0)\right\Vert \leq Cc^{k}.
\]
Therefore%
\[
\left\vert \hat{d}^{k}g(0)(x_{j})\right\vert \leq\frac{1}{\left\vert
f(0)\right\vert }\left\vert \hat{d}^{k}h(0)(x_{j})\right\vert +\frac
{C}{\left\vert f(0)\right\vert }%
%TCIMACRO{\dsum \limits_{l=1}^{k}}%
%BeginExpansion
{\displaystyle\sum\limits_{l=1}^{k}}
%EndExpansion
\frac{k!}{l!(k-l)!}c^{k-l}\left\vert \hat{d}^{l}f(0)(x_{j})\right\vert ,
\]
for every $j\in\mathbb{N}.$ Let $\pi\colon\mathbb{N}\longrightarrow\mathbb{N}$
be an injection. Since $r^{\prime}\geq q^{\prime}$ we have (using Lemma 3.1 of
\cite{Matos-Pellegrino})%
\begin{align*}
&  \left[
%TCIMACRO{\dsum \limits_{j=1}^{\infty}}%
%BeginExpansion
{\displaystyle\sum\limits_{j=1}^{\infty}}
%EndExpansion
\left(  j^{\frac{1}{r^{\prime}}-\frac{1}{q^{\prime}}}\left\vert \hat{d}%
^{k}g(0)(x_{\pi(j)})\right\vert \right)  ^{q^{\prime}}\right]  ^{\frac
{1}{q^{\prime}}}\\
&  \leq\frac{1}{\left\vert f(0)\right\vert }\left[
%TCIMACRO{\dsum \limits_{j=1}^{\infty}}%
%BeginExpansion
{\displaystyle\sum\limits_{j=1}^{\infty}}
%EndExpansion
\left(  j^{\frac{1}{r^{\prime}}-\frac{1}{q^{\prime}}}\left\vert \hat{d}%
^{k}h(0)(x_{\pi(j)})\right\vert \right)  ^{q^{\prime}}\right]  ^{\frac
{1}{q^{\prime}}}\\
&  +\frac{C}{\left\vert f(0)\right\vert }%
%TCIMACRO{\dsum \limits_{l=1}^{k}}%
%BeginExpansion
{\displaystyle\sum\limits_{l=1}^{k}}
%EndExpansion
\frac{k!}{l!(k-l)!}c^{k-l}\left[
%TCIMACRO{\dsum \limits_{j=1}^{\infty}}%
%BeginExpansion
{\displaystyle\sum\limits_{j=1}^{\infty}}
%EndExpansion
\left(  j^{\frac{1}{r^{\prime}}-\frac{1}{q^{\prime}}}\left\vert \hat{d}%
^{l}f(0)(x_{\pi(j)})\right\vert \right)  ^{q^{\prime}}\right]  ^{\frac
{1}{q^{\prime}}}\\
&  \leq\frac{1}{\left\vert f(0)\right\vert }\left\Vert \left(  \hat{d}%
^{k}h(0)(x_{j})\right)  _{j=1}^{\infty}\right\Vert _{(r^{\prime},q^{\prime}%
)}+\frac{C}{\left\vert f(0)\right\vert }%
%TCIMACRO{\dsum \limits_{l=1}^{k}}%
%BeginExpansion
{\displaystyle\sum\limits_{l=1}^{k}}
%EndExpansion
\frac{k!}{l!(k-l)!}c^{k-l}\left\Vert \left(  \hat{d}^{l}f(0)(x_{j})\right)
_{j=1}^{\infty}\right\Vert _{(r^{\prime},q^{\prime})}.
\end{align*}
By Theorem \ref{novoJ} we have%
\begin{align*}
\left\Vert \left(  \hat{d}^{k}h(0)(x_{j})\right)  _{j=1}^{\infty}\right\Vert
_{(r^{\prime},q^{\prime})} &  \leq\left\Vert \hat{d}^{k}h(0)\right\Vert
_{as\left(  (r^{\prime},q^{\prime}),(s^{\prime},p^{\prime})\right)
}\left\Vert (x_{j})_{j=1}^{\infty}\right\Vert _{w,(s^{\prime},p^{\prime})}%
^{k}\leq\left\Vert \hat{d}^{k}h(0)\right\Vert _{as\left(  (r^{\prime
},q^{\prime}),(s^{\prime},p^{\prime})\right)  },\\
\left\Vert \left(  \hat{d}^{l}f(0)(x_{j})\right)  _{j=1}^{\infty}\right\Vert
_{(r^{\prime},q^{\prime})} &  \leq\left\Vert \hat{d}^{l}f(0)\right\Vert
_{as\left(  (r^{\prime},q^{\prime}),(s^{\prime},p^{\prime})\right)
}\left\Vert (x_{j})_{j=1}^{\infty}\right\Vert _{w,(s^{\prime},p^{\prime})}%
^{l}\leq\left\Vert \hat{d}^{l}f(0)\right\Vert _{as\left(  (r^{\prime
},q^{\prime}),(s^{\prime},p^{\prime})\right)  },
\end{align*}
for all $l=1,\ldots,k.$ Thus%
\begin{align*}
&  \left[
%TCIMACRO{\dsum \limits_{j=1}^{\infty}}%
%BeginExpansion
{\displaystyle\sum\limits_{j=1}^{\infty}}
%EndExpansion
\left(  j^{\frac{1}{r^{\prime}}-\frac{1}{q^{\prime}}}\left\vert \hat{d}%
^{k}g(0)(x_{\pi(j)})\right\vert \right)  ^{q^{\prime}}\right]  ^{\frac
{1}{q^{\prime}}}\\
&  \leq\frac{1}{\left\vert f(0)\right\vert }\left\Vert \hat{d}^{k}%
h(0)\right\Vert _{as\left(  (r^{\prime},q^{\prime}),(s^{\prime},p^{\prime
})\right)  }+\frac{C}{\left\vert f(0)\right\vert }%
%TCIMACRO{\dsum \limits_{l=1}^{k}}%
%BeginExpansion
{\displaystyle\sum\limits_{l=1}^{k}}
%EndExpansion
\frac{k!}{l!(k-l)!}c^{k-l}\left\Vert \hat{d}^{l}f(0)\right\Vert _{as\left(
(r^{\prime},q^{\prime}),(s^{\prime},p^{\prime})\right)  }.
\end{align*}
We conclude that$\left(  \left\vert \hat{d}^{k}g(0)(x_{\pi(j)})\right\vert
\right)  _{j=1}^{\infty}\in\ell_{(r^{\prime},q^{\prime})}$ and thus $\left(
\hat{d}^{k}g(0)(x_{\pi(j)})\right)  _{j=1}^{\infty}\in c_{0}$. Since $\pi$ is
arbitrary we have%
\[
\left[
%TCIMACRO{\dsum \limits_{j=1}^{\infty}}%
%BeginExpansion
{\displaystyle\sum\limits_{j=1}^{\infty}}
%EndExpansion
\left(  j^{\frac{1}{r^{\prime}}-\frac{1}{q^{\prime}}}\left\vert \hat{d}%
^{k}g(0)(x_{\sigma(j)})\right\vert \right)  ^{q^{\prime}}\right]  ^{\frac
{1}{q^{\prime}}}=\left\Vert \left(  \hat{d}^{k}g(0)(x_{j})\right)
_{j=1}^{\infty}\right\Vert _{(r^{\prime},q^{\prime})}%
\]
for some injection $\sigma\colon\mathbb{N}\longrightarrow\mathbb{N}$.\newline
Now, let $0\neq(x_{j})_{j=1}^{\infty}\in\ell_{\left(  s^{\prime},p^{\prime
}\right)  }^{w}(F).$ Defining
\[
y_{j}=\frac{x_{j}}{\left\Vert (x_{j})_{j=1}^{\infty}\right\Vert _{w,(s^{\prime
},p^{\prime})}},
\]
for all $j\in\mathbb{N},$ we have $\left\Vert (y_{j})_{j=1}^{\infty
}\right\Vert _{w,(s^{\prime},p^{\prime})}=1$ and using the previous estimates
we obtain%
\begin{align*}
&  \left\Vert \left(  \hat{d}^{k}g(0)(x_{j})\right)  _{j=1}^{\infty
}\right\Vert _{(r^{\prime},q^{\prime})}\\
&  \leq\left(  \frac{1}{\left\vert f(0)\right\vert }\left\Vert \hat{d}%
^{k}h(0)\right\Vert _{as\left(  (r^{\prime},q^{\prime}),(s^{\prime},p^{\prime
})\right)  }+\frac{C}{\left\vert f(0)\right\vert }%
%TCIMACRO{\dsum \limits_{l=1}^{k}}%
%BeginExpansion
{\displaystyle\sum\limits_{l=1}^{k}}
%EndExpansion
\frac{k!}{l!(k-l)!}c^{k-l}\left\Vert \hat{d}^{l}f(0)\right\Vert _{as\left(
(r^{\prime},q^{\prime}),(s^{\prime},p^{\prime})\right)  }\right)  \left\Vert
(x_{j})_{j=1}^{\infty}\right\Vert _{w,(s^{\prime},p^{\prime})}^{k}.
\end{align*}
Using again Theorem \ref{novoJ}, it follows that $\hat{d}^{k}g(0)\in
\mathcal{P}_{as\left(  \left(  r^{\prime},q^{\prime}\right)  ;\left(
s^{\prime},p^{\prime}\right)  \right)  }\left(  ^{k}F\right)  .$ Since $f,h\in
Exp_{as\left(  \left(  r^{\prime},q^{\prime}\right)  ;\left(  s^{\prime
},p^{\prime}\right)  \right)  }\left(  F\right)  ,$ there are $A,B\geq0$ and
$a,b>0$ such that%
\[
\left\Vert \hat{d}^{k}h(0)\right\Vert _{as\left(  (r^{\prime},q^{\prime
}),(s^{\prime},p^{\prime})\right)  }\leq Aa^{k}%
\]
and%
\[
\left\Vert \hat{d}^{l}f(0)\right\Vert _{as\left(  (r^{\prime},q^{\prime
}),(s^{\prime},p^{\prime})\right)  }\leq Bb^{l},
\]
for $l=1,\ldots,k.$ Hence%
\begin{align*}
\left\Vert \left(  \hat{d}^{k}g(0)(x_{j})\right)  _{j=1}^{\infty}\right\Vert
_{(r^{\prime},q^{\prime})} &  \leq\left(  \frac{Aa^{k}}{\left\vert
f(0)\right\vert }+\frac{CB}{\left\vert f(0)\right\vert }%
%TCIMACRO{\dsum \limits_{l=1}^{k}}%
%BeginExpansion
{\displaystyle\sum\limits_{l=1}^{k}}
%EndExpansion
k!\frac{1}{l!}\frac{1}{(k-l)!}c^{k-l}b^{l}\right)  \left\Vert (x_{j}%
)_{j=1}^{\infty}\right\Vert _{w,(s^{\prime},p^{\prime})}^{k}\\
&  \leq\left(  \frac{A}{\left\vert f(0)\right\vert }+\frac{CB}{\left\vert
f(0)\right\vert }\right)  \left(  a+b+c\right)  ^{k}\left\Vert (x_{j}%
)_{j=1}^{\infty}\right\Vert _{w,(s^{\prime},p^{\prime})}^{k},
\end{align*}
for all $(x_{j})_{j=1}^{\infty}\in\ell_{\left(  s^{\prime},p^{\prime}\right)
}^{w}(F)$ and by the definition of $\left\Vert \cdot\right\Vert _{as\left(
(r^{\prime},q^{\prime}),(s^{\prime},p^{\prime})\right)  }$ (see Theorem
\ref{novoJ}) we have%
\[
\left\Vert \hat{d}^{k}g(0)\right\Vert _{as\left(  (r^{\prime},q^{\prime
}),(s^{\prime},p^{\prime})\right)  }\leq Dd^{k},
\]
with $D=\frac{A+BC}{\left\vert f(0)\right\vert }\geq0$ and $d=a+b+c\geq
0.$\newline Now, suppose that $f(0)=0$ and define
\[
f_{0}\left(  x\right)  =f\left(  x\right)  +\psi\left(  x\right)
\]
and
\[
h_{0}\left(  x\right)  =h\left(  x\right)  +\psi\left(  x\right)  g\left(
x\right)
\]
for all $x\in F,$ where $\psi\in Exp_{as\left(  \left(  r^{\prime},q^{\prime
}\right)  ;\left(  s^{\prime},p^{\prime}\right)  \right)  }\left(  F\right)
,$ $\psi\left(  0\right)  \neq0$ and $\psi$ is non constant (for example, let
$\psi\left(  x\right)  =1+P\left(  x\right)  ,$ with $P\in\mathcal{P}%
_{as\left(  \left(  s,p\right)  ;\left(  r,q\right)  \right)  }\left(
^{n}F\right)  $ for some $n\neq0$). Thus $h_{0}=f_{0}g$, $f_{0}\left(
0\right)  \neq0,$ $f_{0}\in Exp_{as\left(  \left(  r^{\prime},q^{\prime
}\right)  ;\left(  s^{\prime},p^{\prime}\right)  \right)  }\left(  F\right)  $
and by Lemma \ref{prod_funcoes}, $h_{0}\in Exp_{as\left(  \left(  r^{\prime
},q^{\prime}\right)  ;\left(  s^{\prime},p^{\prime}\right)  \right)  }\left(
F\right)  .$ Applying the result we have just proved, it follows that $g\in
Exp_{as\left(  \left(  r^{\prime},q^{\prime}\right)  ;\left(  s^{\prime
},p^{\prime}\right)  \right)  }\left(  F\right)  .$
\end{proof}

For $F=E^{\prime},$ we have that $Exp_{as\left(  \left(  r^{\prime},q^{\prime
}\right)  ;\left(  s^{\prime},p^{\prime}\right)  \right)  }\left(  E^{\prime
}\right)  $ is closed under division, and so it follows from Theorems
\ref{teoremaAproximacao1} and \ref{Teorema de existencia} that we have
existence and approximation results for convolution equations defined on
$\mathcal{H}_{\widetilde{N}b,\left(  \left(  r,q\right)  ;\left(  s,p\right)
\right)  }(E)$ as enunciated below.

\begin{theorem}
The vector subspace of $\mathcal{H}_{\widetilde{N}b,\left(  \left(
r,q\right)  ;\left(  s,p\right)  \right)  }(E)\textcolor{red}{,}$ generated by the exponential
polynomial solutions of the homogeneous equation $L=0,$ is dense in the closed
subspace of all solutions of the homogeneous equation, that is, the vector
subspace of $\mathcal{H}_{\widetilde{N}b,\left(  \left(  r,q\right)  ;\left(
s,p\right)  \right)  }(E)$ generated by
\[
\mathcal{L=}\left\{  P\exp\varphi;P\in\mathcal{P}_{\widetilde{N},\left(
\left(  r,q\right)  ;\left(  s,p\right)  \right)  }\left(  ^{m}E\right)
,m\in\mathbb{N}_{0},\varphi\in E^{\prime},L\left(  P\exp\varphi\right)
=0\right\}
\]
is dense in
\[
\ker L=\left\{  f\in\mathcal{H}_{\widetilde{N}b,\left(  \left(  r,q\right)
;\left(  s,p\right)  \right)  }(E);Lf=0\right\}  .
\]

\end{theorem}

\begin{theorem}
If $L$ is a non zero convolution operator, then
\[
L\left(  \mathcal{H}_{\widetilde{N}b,\left(  \left(  r,q\right)  ;\left(
s,p\right)  \right)  }\left(  E\right)  \right)  =\mathcal{H}_{\widetilde
{N}b,\left(  \left(  r,q\right)  ;\left(  s,p\right)  \right)  }\left(
E\right)  .
\]

\end{theorem}

%\subsection{Final Remarks}
%We showed the utility of the technique in two situations, but it is clear that it can be applied in many others. ......

\textbf{Acknowledgements. }The authors thank M\'{a}rio C. Matos for
introducing them to this subject and for his permanent encouragement. The
authors also want to thank G. Botelho for his useful suggestions.

\bigskip

\noindent Vin\'{\i}cius V. F\'{a}varo: \newline Faculdade de Matem\'atica,
Universidade Federal de Uberl\^andia, Uberl\^andia, MG, Cep: 38.400-902,
Brazil, email: vvfavaro@gmail.com

\bigskip

\noindent Daniel Pellegrino:\newline Departamento de Matem\'{a}tica,
Universidade Federal da Para\'{\i}ba, Jo\~{a}o Pessoa, PB, Cep: 58.051--900,
Brazil, email: dmpellegrino@gmail.com


\begin{thebibliography}{99}                                                                                               %
%\bibitem {ansari}S. I. Ansari,, \textit{Hypercyclic and cyclic vectors}, J.
%Funct. Anal. \textbf{128} (1993), 374--383.


\bibitem {aronmarkose}R. Aron and D. Markose, \textit{On universal functions},
in: Satellite Conference on Infinite Dimensional Function Theory, J. Korean
Math. Soc. \textbf{41} (2004), 65--76.

\bibitem {bay}F. Bayart, \'E. Matheron, \textit{Dynamics of linear operators},
Cambridge Tracts in Mathematics, 179. Cambridge University Press, Cambridge, 2009.

\bibitem {bay2}F. Bayart, \'E. Matheron, \textit{Mixing operators and small
subsets of the circle}, J. Reine Angew. Math., to appear.

\bibitem {BPS1}L. Bernal-Gonzalez, D. Pellegrino, J.B. Seoane-Sepulveda,
\textit{Linear subsets of nonlinear sets in topological vector spaces}, Bull.
Amer. Math. Soc. (N.S.) \textbf{51 } (2014), no. 1, 71--130.

\bibitem {BBFJ}F. J. Bertoloto, G. Botelho, V. V. F\'avaro, A. M. Jatob\'a,
\textit{Hypercyclicity of convolution operators on spaces of entire
functions}. Ann. Inst. Fourier (Grenoble) \textbf{63 } (2013), 1263-1283.

%\bibitem {bes1999}J. B\`es and A. Peris \textit{Hereditarily hypercyclic
%operators}, J. Funct. Anal. \textbf{167} (1999), 94--112.


\bibitem {bes2012}J. B\`es, \"O. Martin, A. Peris and S. Shkarin
\textit{Disjoint mixing operators}, J. Funct. Anal. \textbf{263} (2012), 1283--1322.

\bibitem {birkhoff}G. D. Birkhoff, \textit{D\'emonstration d'un th\'eor\`eme
\'el\'ementaire sur les fonctions enti\`eres}, C. R. Acad. Sci. Paris
\textbf{189} (1929), 473--475.

\bibitem {boyd1}C. Boyd, \textit{Duality and reflexivity of spaces of
approximable polynomials on locally convex spaces}, Monatsh. Math.
\textbf{130} (2000), 177--188.

\bibitem {boyd2}C. Boyd and A. Brown, \textit{Duality in spaces of polynomials
pf degree at most $n$}, J. Math. Anal. Appl. \textbf{429} (2015), 1271--1290.

\bibitem {boyd3}C. Boyd S. Dineen and P. Rueda, \textit{Locally Asplund space
of holomorphic functions}, Michigan Math. J. \textbf{50} (2002), 493--506.

\bibitem {CDjmaa}D. Carando and V. Dimant, \textit{Duality in spaces of
nuclear and integral polynomials}, J. Math. Anal. Appl. \textbf{241} (2000), 107--121.

\bibitem {CDSjmaa}D. Carando, V. Dimant and S. Muro, \textit{Hypercyclic
convolution operators on Fr\'echet spaces of analytic functions}, J. Math.
Anal. Appl. \textbf{336} (2007), 1324--1340.

\bibitem {chan}K. C. Chan and J. H. Shapiro, \textit{The cyclic behaviour of
translation operators on Hilbert spaces of entire funcitons}, Indiana Univ. J.
\textbf{40} (1991), 1421--1449.

\bibitem {Col-Mat}J. F. Colombeau and M. C. Matos, \textit{Convolution
equations in spaces of infinite dimensional entire functions}, Indag. Math.
\textbf{42} (1980), 375-389.

\bibitem {cp}J.F Colombeau and B. Perrot, \textit{Convolution equations in
spaces of infinite dimensional entire functions of exponencial and related
types}, Trans. Amer. Math. Soc. \textbf{258} (1980), 191-198.

\bibitem {cgp}J.F Colombeau, R. Gay and B. Perrot, \textit{Division by
holomorphic functions and convolution equations in infinite dimension}, Trans.
Amer. Math. Soc. \textbf{264} (1981), 381-391.

\bibitem {defant}A. Defant and K. Floret, \textit{Tensor Norms and Operator
Ideals}, North-Holland Math. Studies \textbf{176}, 1993.

\bibitem {Diestel}J. Diestel, H. Jarchow and A. Tonge, \textit{Absolutely
Summing Operators}, Cambridge Studies in Advanced Mathematics \textbf{43}, 1995.

\bibitem {Dineen}S. Dineen, \textit{Complex Analysis on Infinite Dimensional
Spaces}, Springer-Verlag, London, 1999.

\bibitem {Dineen-70}S. Dineen, \textit{Holomorphy types on a Banach space},
Studia Math. \textbf{39} (1971), 241--288.

\bibitem {DS}J. Dieudonn\'e and L. Schwartz, \textit{La dualit\'e dans les
espaces ($\mathcal{F}$) et ($\mathcal{LF}$),} Ann. Inst. Fourier (Grenoble)
\textbf{I} (1949), 61--101.

\bibitem {DW}T.A.W. Dwyer III, \textit{Convolution equations for vector-valued
entire functions of nuclear bounded type}, Trans. Amer. Math. Soc.
\textbf{217} (1976), 105- 119.

\bibitem {DW2}T.A.W. Dwyer III, \textit{Partial differential equations in
Fischer-Fock spaces for the Hilbert-Schmidt holomorphy type}, Bull. Amer.
Math. Soc. \textbf{77} (1971), 725- 730.

\bibitem {Fa}V. V. F\'{a}varo,\textit{The Fourier-Borel transform between
spaces of entire functions of a given type and order}, Portugal. Math.
\textbf{65} (2008), 285-309.

\bibitem {FaBelg}V. V. F\'{a}varo,\textit{Convolution equations on spaces of
quasi-nuclear functions of a given type and order}, Bull. Belg. Math. Soc.
Simon Stevin \textbf{17} (2010), 535--569.

\bibitem {favaro-jatoba1}\textrm{V. V. F\'avaro, A. M. Jatob\'a}, \textit{
Holomorphy types and spaces of entire functions of bounded type on Banach
spaces}, Czech. Math. Journal \textbf{59} (2009), 909-927.

\bibitem {favaro-jatoba2}\textrm{V. V. F\'avaro, A. M. Jatob\'a},
\textit{Holomorphy types and the Fourier-Borel transform between spaces of
entire functions of a given type and order defined on Banach spaces}, Math.
Scand. \textbf{110} (2012), 111-139.

\bibitem {favaromatospellegrino}V. V. F\'{a}varo, M. C. Matos, D. Pellegrino,
\textit{On Lorentz nuclear homogeneous polynomials between Banach spaces},
Portugal. Math. \textbf{67} (2010), 413-435.

\bibitem {FM}\textrm{V. V. F\'avaro, J. Mujica}, \textit{ Hypercyclic
convolution operators on spaces of entire functions}, to appear in J. Operator Theory.

\bibitem {gethner}R. M. Gethner and J. H. Shapiro, \textit{Universal vectors
for operators on spaces of holomorphic functions}, Proc. Amer. Math. Soc.
\textbf{100} (1987), 281--288.

\bibitem {godefroy}G. Godefroy and J. H. Shapiro, \textit{Operators with
dense, invariant, cyclic vector manifolds}, J. Funct. Anal. \textbf{98}
(1991), 229--269.

\bibitem {goswinBAMS}K. G. Grosse-Erdmann, \textit{Universal families and
hypercyclic operators}, Bull. Amer. Math. Soc. \textbf{36} (1999), 345--381.

\bibitem {G}C. P. Gupta, Convolution Operators and Holomorphic Mappings on a
Banach Space, Seminaire d'Analyse Moderne, 2, Universit\'e de Sherbrooke,
Sherbrooke, 1969.

\bibitem {Gupta}C. P. Gupta, \textit{On the Malgrange Theorem for nuclearly
entire functions of bounded type on a Banach space}, Indag. Math. \textbf{32}
(1970), 356-358.

\bibitem {hhh}L. H\"ormander, \textit{On the division of distributions by
polynomials}, Ark. Math. \textbf{3} (1958), 555-568.

\bibitem {kitai}C. Kitai, \textit{Invariant closed sets for linear operators},
Dissertation, University of Toronto, 1982.

\bibitem {loja}S. Lojasiewicz, \textit{Sur le probl\`{e}me de la division},
Studia Math. \textbf{18} (1959), 87-136.

\bibitem {maclane}G. R. MacLane, \textit{Sequences of derivatives and normal
families}, J. Anal. Math. \textbf{2} (1952), 72--87.

\bibitem {Malgrange}B. Malgrange, \textit{Existence et approximation des
solutions des \'{e}quations aux deriv\'{e}es partielles et des \'{e}quations
des convolutions}, Ann. Inst. Fourier (Grenoble) \textbf{6 }(1955-56), 271-355.

\bibitem {Martineau}A. Martineau, \textit{\'{E}quations diff\'{e}rentielles
d'ordre infini}, Bull. Soc. Math. France \textbf{95} (1967), 109-154.

\bibitem {Matos-F}M. C. Matos, \textit{Sur le th\'{e}or\`{e}me d'approximatin
et d'existence de Malgrange-Gupta}, C. R. Acad. Sci. Paris,\textbf{ 271}
(1970), 1258-1259.

\bibitem {Matos-Z}M. C. Matos, \textit{On Malgrange Theorem for nuclear
holomorphic functions in open balls of a Banach space}, Math. Z. \textbf{171}
(1980), 113-123.

\bibitem {Matos-Z2}M. C. Matos, \textit{Correction to \textquotedblleft On
Malgrange Theorem for nuclear holomorphic functions in open balls of a Banach
space\textquotedblright}, Math. Z. (1980), 289-290.

\bibitem {Matos-C2}M. C. Matos, \textit{On convolution operators in spaces of
entire functions of a given type and order}, in: Complex Analysis, Functional
Analysis and Approximation Theory (J. Mujica, ed.), pp. 129-171. North-Holland
Math. Studies \textbf{125}, North-Holland, Amsterdam, 1986.

%\bibitem {Matos-Espanha}M. C. Matos, \textit{On multilinear mappings of
%nuclear type}, Rev. Mat. Univ. Compl. Madrid, \textbf{6} (1993), 61-81.


\bibitem {Matos-livro}M. C. Matos, \textit{Absolutely summing mappings,
nuclear mappings and convolution equations}, IMECC-UNICAMP, 2007. Web:
http://www.ime.unicamp.br/$\sim$matos.

\bibitem {MN}M. C. Matos and L. Nachbin, \textit{Entire functions on locally
convex spaces and convolution operators}, Comp. Math. \textbf{44} (1981), 145-181.

\bibitem {Matos-Pellegrino}M. C. Matos and D. Pellegrino, \textit{Lorentz
summing mappings}, Math. Nachr. \textbf{283} (2010), 1409-1427.

\bibitem {Mujica}J. Mujica, \textit{Complex analysis in Banach spaces},
North-Holland Mathematics Studies 120, North-Holland, 1986.

\bibitem {MPS}S. Muro, D. Pinasco, M. Savransky, \textit{Strongly mixing
convolution operators on Fr\'echet spaces of holomorphic functions}, Integr.
Equ. Oper. Theory \textbf{80} (2014), 453-468.

\bibitem {Nach-B}L. Nachbin, \textit{Recent developments in infinite
dimensional holomorphy}, Bull. Amer. Math. Soc. \textbf{79} (1973), 625-639.

\bibitem {nachbin}L. Nachbin, \textit{Topology on Spaces of Holomorphic
Mappings,} Springer-Verlag, New York, 1969.

\bibitem {peterssonjmaa}H. Petersson, \textit{Hypercyclic subspaces for
Fr\'echet space operators}, J. Math. Anal. Appl. \textbf{319} (2006), 764--782.

\bibitem {Pietsch}A. Pietsch, \textit{Operator Ideals}, North-Holland, 1980.

\bibitem {pietsch}\textrm{A. Pietsch}, \textit{Ideals of multilinear
functionals}. In: Proceedings of the Second International Conference on
Operator Algebras, Ideals and Their Applications in Theoretical Physics,
pp.185-199, Teubner, $1983$.

\bibitem {SHA}H. H. Schaefer, \textit{Topological Vector Spaces},
Springer-Verlag, 1971.

%\bibitem {shkarin}S. Shkarin, \textit{Remarks on common hypercyclic vectors},
%J. Funct. Anal.\textbf{258} (2010), 132--160.

\end{thebibliography}
\end{document}